\documentclass[reqno]{amsart}
\usepackage{amsmath, amssymb, amsthm, epsfig}
\usepackage{hyperref, latexsym}
\usepackage{url}
\usepackage[mathscr]{euscript}
\usepackage{enumerate}
\usepackage{color}
\usepackage{fullpage} 
\usepackage{setspace}
\usepackage{graphicx} 
\usepackage{mathabx}

\allowdisplaybreaks

\def\today{\ifcase\month\or
  January\or February\or March\or April\or May\or June\or
  July\or August\or September\or October\or November\or December\fi
  \space\number\day, \number\year}

 \newtheorem{theorem}{Theorem}
  
 \newtheorem{lemma}[theorem]{Lemma}
 \newtheorem{proposition}[theorem]{Proposition}
 \newtheorem{corollary}[theorem]{Corollary}
 \theoremstyle{definition}

 \theoremstyle{remark}

 \newcommand{\mc}{\mathcal}

 \newcommand{\J}{\mc{J}}

 \newcommand{\C}{\mathbb{C}}
 \newcommand{\R}{\mathbb{R}}
 \newcommand{\N}{\mathbb{N}}
 
 \newcommand{\Z}{\mathbb{Z}}

 \newcommand{\ds}{\text{\rm d}s}

  \renewcommand{\d}{\text{\rm d}}

 \newcommand{\dx}{\text{\rm d}x}
 
 \newcommand{\dy}{\text{\rm d}y}

\newcommand{\ov}{\overline}

\begin{document}
\title[Hilbert spaces and low-lying zeros of $L$-functions]{Hilbert spaces and low-lying zeros of $L$-functions}
\author[Carneiro, Chirre, and Milinovich]{Emanuel Carneiro, Andr\'{e}s Chirre and Micah B. Milinovich}

\address{
ICTP - The Abdus Salam International Centre for Theoretical Physics,
Strada Costiera, 11, I - 34151, Trieste, Italy.}

\email{carneiro@ictp.it}

\address{Department of Mathematical Sciences, Norwegian University of Science and Technology, NO-7491 Trondheim, Norway.}

\email{carlos.a.c.chavez@ntnu.no }

\address{Department of Mathematics, University of Mississippi, University, MS 38677 USA.}

\email{mbmilino@olemiss.edu}

\allowdisplaybreaks
\numberwithin{equation}{section}

\maketitle  

\begin{abstract}
Generalizing previous work of Iwaniec, Luo, and Sarnak (2000), we use information from one-level density theorems to estimate the proportion of non-vanishing of $L$-functions in a family at a low-lying height on the critical line (measured by the analytic conductor). To solve the Fourier optimization problems that arise, we provide a unified framework based on the theory of reproducing kernel Hilbert spaces of entire functions (there is one such space associated to each symmetry type). Explicit expressions for the reproducing kernels are given. We also revisit the problem of estimating the height of the first low-lying zero in a family, considered by Hughes and Rudnick (2003) and Bernard (2015). We solve the associated Fourier optimization problem in this setting by establishing a connection to the theory of de Branges spaces of entire functions and using the explicit reproducing kernels. In an appendix, we study the related problem of determining the sharp embeddings between the Hilbert spaces associated to the five symmetry types and the classical Paley-Wiener space.
\end{abstract}


\section{Introduction}
A central topic in number theory is to understand the distribution of zeros of $L$-functions. In particular, a great deal of effort has gone into proving that many $L$-functions in a family cannot simultaneously vanish at a given point. For any $t>0$, it is generally believed that at most one principal automorphic $L$-function with unitary central character can vanish at the point $s=\frac{1}{2}+it$ on the critical line. This belief is a consequence of the so-called Grand Simplicity hypothesis, e.g.~\cite{RS}, which asserts that the multi-set of positive ordinates of zeros of all principal automorphic $L$-functions are linearly independent over $\mathbb{Q}$. In this paper, we use results on 1-level density for low-lying zeros of families of $L$-functions and the solution of a certain extremal problem involving entire functions of exponential type, to study the non-vanishing of $L$-functions at low-lying heights on the critical line (where the height is measured in terms of the analytic conductor). This is a generalization of a problem considered by Iwaniec, Luo, and Sarnak in \cite[Appendix A]{ILS}, who were interested in using 1-level density results to study the non-vanishing of $L$-functions in families at the central point. We solve our extremal problem using different methods, appealing to the framework of reproducing kernel Hilbert spaces of entire functions developed by Carneiro, Chandee, Littmann, and Milinovich in \cite{CCLM}.

\smallskip

In a complementary direction, we also address here the problem of estimating the height of the first low-lying zero in a family of $L$-functions. This was first considered by Hughes and Rudnick \cite[Theorem 8.1]{HR} in the context of Dirichlet $L$-functions, and is connected to a different extremal problem in analysis. The solution of the corresponding extremal problem for other families of $L$-functions was later obtained in the impressive work of Bernard \cite{Bernard}, by means of a delicate analysis of an associated Volterra differential equation in connection to  the classical bases of Chebyshev polynomials. We provide here an alternative approach to solve this extremal problem for all families, obtaining it as a corollary of a more general result within the rich theory of de Branges spaces of entire functions \cite{Branges}. We conclude the paper with an appendix in which we determine the sharp embeddings between the Hilbert spaces naturally associated to families of $L$-functions and the classical Paley-Wiener space.

\subsection{Dirichlet $L$-functions} As an illustration of our more general results in Theorem \ref{Thm2_Non-vanishing}, we first consider the family of primitive Dirichlet $L$-functions modulo a prime $q$.  

\begin{theorem} \label{Dirichlet1}
Let $q$ be prime and assume the generalized Riemann hypothesis \textup{(GRH)} for Dirichlet $L$-functions modulo $q$. Then, for any fixed $t>0$, we have 
\[
\frac{1}{q\!-\!2}\sum_{\substack{\chi \, (\textup{mod }q) \\ \chi \ne \chi_0}} \ \underset{s=\frac{1}{2}}{\textup{ord}} \ L\!\left(s+ \frac{2 \pi i t }{\log q}\,,\,\chi\right) \ \le \ \frac{1}{2} \, \left(1+\left|\frac{\sin 4 \pi t }{4\pi t}\right|\right)^{\!-1} + \ O\!\left(\frac{1}{\log q}\right),
\]
where $\chi_0$ denotes the principal character \textup{(mod $q$)}. Hence, for any $\varepsilon>0$, the proportion of primitive Dirichlet characters $\chi$ \textup{(mod $q$)} for which $L\!\left( \frac{1}{2}+ \frac{2 \pi i t }{\log q}\,,\,\chi\right) \ne 0$
is at least $1-\frac{1}{2} \, \left(1+\left|\frac{\sin 4 \pi t }{4\pi t}\right|\right)^{\!-1} - \varepsilon$ when $q$ is large.
\end{theorem}

Assuming GRH, Murty \cite{MU} proved that the proportion of primitive Dirichlet characters $\chi$ (mod $q$) for which $L(\frac{1}{2},\chi)\ne 0$ is at least $\frac{1}{2}-\varepsilon$ for any $\varepsilon>0$.  Our result in Theorem \ref{Dirichlet1} always gives at least as good of a proportion of non-vanishing, and in some cases a significantly larger proportion of non-vanishing, at every fixed low-lying height on the critical line. For example, assuming GRH, Theorem \ref{Dirichlet1} implies that the proportion of primitive Dirichlet characters $\chi$ (mod $q$) for which $L\Big( \frac{1}{2}+ \frac{ \pi i }{4\log q}\,,\,\chi\Big) \ne 0$ is at least $\frac{4+\pi}{4+2\pi} - \varepsilon = 0.69449\ldots$. In Figure \ref{figure1}, we plot the proportion of non-vanishing implied by Theorem \ref{Dirichlet1}. The proportion tends to $3/4$ as $t \to 0^+$ (note, however, that our proof requires $t>0$). If $t=0$, then our method simply recovers Murty's result\footnote{There is no ``discontinuity" hidden in the method here. Looking from the vanishing side, one may argue that the $1/2$-vanishing at the central point would split into a $1/4$-vanishing at $t$ and a $1/4$-vanishing at $-t$, for $|t|$ small.}, and we comment more on this later. Also note the proportion in Theorem \ref{Dirichlet1} tends to $\frac{1}{2}$ as $t\to +\infty$, implying that our result only improves upon Murty's proportion for low-lying heights. 

\smallskip

\begin{figure} 
\includegraphics[scale=.7]{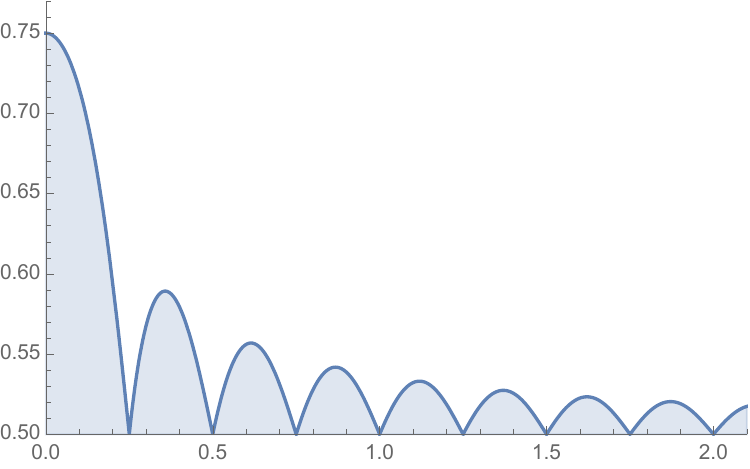} 
\caption{A plot of the function $t \mapsto 1-\frac{1}{2} \left(1+\left|\frac{\sin 4 \pi t }{4\pi t}\right|\right)^{\!-1}$ for small $t>0$. }
\label{figure1}
\end{figure}

The set of primitive Dirichlet $L$-functions modulo $q$ is an example of a family of $L$-functions with unitary symmetry. In \S \ref{SSec:Examples}, below, we give additional examples of families of $L$-functions with other symmetry types that also fall under the scope of our general machinery described in Theorem \ref{Thm2_Non-vanishing}. Perhaps of particular interest are examples of families of $L$-functions with an odd functional equation (and odd orthogonal symmetry). In this case, the $L$-functions are guaranteed to vanish at $s=\frac{1}{2}$ and the phenomenon of `zero-repulsion' suggests that it should be unlikely that a large proportion within a family also vanishes at low-lying heights on the critical line. Our results support this. 

\smallskip

There are several unconditional results concerning the proportion of non-vanishing of the family of primitive Dirichlet $L$-functions (mod $q$) at the central point, for instance \cite{BM, B,IS,KN,KMN}. In the spirit of Theorem \ref{Dirichlet1}, it would be interesting to see if mollifier methods can be used to unconditionally give an increased proportion of non-vanishing in this family at low-lying heights, without appealing to GRH.

\subsection{A Hilbert space framework for estimates of non-vanishing of $L$-functions in families}

\subsubsection{1-level density and symmetry groups} 

Let $\mathcal{F}$ be a family of automorphic objects. For each $f \in \mathcal{F}$, let 
\[
L(s,f)=\sum_{n=1}^\infty \lambda_f(n) \, n^{-s}
\] 
be the associated $L$-function and we assume that $L(s,f)$ admits an analytic continuation to an entire function. We denote the non-trivial zeros of $L(s,f)$ by $\rho_f=\frac{1}{2}+i\gamma_f$ and we work under the assumption of the generalized Riemann hypothesis (GRH) for such families. This means that $\gamma_f \in \R$. Broadly speaking, we are interested in the distribution of the ordinates $\gamma_f$ of the low-lying zeros of $L(s,f)$, i.e.~the distribution of zeros close to the central point $s=\frac{1}{2}$ (measured by the analytic conductor), as $f$ varies over the family $\mathcal{F}$. 

\smallskip

For $f \in \mathcal{F}$, we assume that there is a completed $L$-function $\Lambda(s,f)=L_\infty(s,f) \, L(s,f)$ which satisfies a functional equation of the form
\begin{equation*}
\Lambda(s,f)= \varepsilon_f \, \Lambda\big(1-s,\bar{f} \, \big),
\end{equation*}
where $|\varepsilon_f|=1$ and $L\big(s,\bar{f} \, \big)$ is the dual $L$-function with Dirichlet series coefficients $\lambda_{\bar{f}}(n) = \overline{\lambda_f(n)}$. Throughout this paper, we always assume that our family $\mathcal{F}$ satisfies the following assumption:
\begin{equation} \label{assumption}
\text{if $f \in \mathcal{F}$, then $\bar{f} \in \mathcal{F}$}.
\end{equation}
We also assume that no zeros of $L_\infty(s,f)$ are on the line $\mathrm{Re}(s)=1/2$, in which case the functional equation implies that
\begin{equation} \label{duality}
\underset{s=\frac{1}{2} + it}{\textup{ord}} \, L\!\left(s, f\right) = \underset{s=\frac{1}{2} - it}{\textup{ord}} \, L\!\left(s,\bar{f}\,\right).
\end{equation}
This observation is key to the setup of the optimization problem that is used to prove Theorem \ref{Thm2_Non-vanishing}. If $L(s,f)= L\big(s,\bar{f} \, \big)$, then we say $L(s,f)$ is self-dual and in this case we assume that $\varepsilon_f = \pm 1$.  If $\varepsilon_f = 1$, then we say that the functional equation is even. If $\varepsilon_f = - 1$, then we say the functional equation is odd. 

\smallskip

The density conjecture of Katz and Sarnak \cite{KS1,KS2} asserts that for each natural family $\{L(s,f), \ f \in \mathcal{F}\}$ of $L$-functions there is an associated symmetry group $G=G(\mathcal{F}$), where $G$ is either: unitary ${\rm U}$, symplectic ${\rm Sp}$, orthogonal ${\rm O}$, even orthogonal ${\rm SO}(\rm{even})$, or odd orthogonal ${\rm SO}({\rm odd}).$\footnote{In the literature one finds small variations of this notation. Here we choose to follow the notation in \cite{ILS} for such groups.} We wish to consider averages over $f \in \mathcal{F}$, ordered by the conductor. Following the notation in \cite{ILS}, we let $\mathcal{F}(Q)$ denote either of the finite sets
\[
\{  f \in \mathcal{F} : c_f = Q \} \quad \text{ or } \quad \{  f \in \mathcal{F} : c_f \le Q \}
\]
as $Q \to \infty$, where $c_f$ denotes the analytic conductor of $L(s,f)$ and let $|\mathcal{F}(Q)|$ denote its cardinality. In the examples listed below, which of these two sets we consider is clear from context. If $\phi : \mathbb{R} \to \mathbb{R}$ is a smooth function whose Fourier transform has compact support, then Katz and Sarnak conjecture that
\begin{equation}\label{20210430_11:41}
\lim_{Q \to\infty} \, \frac{1}{|\mathcal{F}(Q)|} \sum_{f\in \mathcal{F}(Q)} \sum_{\gamma_f} \, \phi\!\left( \gamma_f \frac{\log c_f}{2\pi}\right) = \int_\mathbb{R} \phi(x) \, W_G(x) \, \d x,
\end{equation}
where the sum over $\gamma_f$ counts multiplicity and $W_G$ is a function (or density) depending on the symmetry group $G$ of $\mathcal{F}$. This is the so-called 1-level density of the low-lying zeros of the family; see the introduction of \cite{ILS} or the survey article \cite{KS1} for a more detailed discussion and the connection to random matrix theory. Since $\phi$ is expected to have some decay at infinity and hence localizes the sum near the origin, the zeros that are within $O(1/\log c_f)$ of the central point contribute more significantly. For these five symmetry groups, Katz and Sarnak determined the density functions:
\begin{equation}\label{densities}
\begin{split}
W_{ \rm U}(x) \, &= \, 1\,;
\\
W_{\rm Sp}(x) \, &= \, 1- \frac{\sin 2\pi x}{2\pi x}\,;
\\
W_{\rm O}(x) \, &= \, 1 + \tfrac{1}{2} \boldsymbol{\delta}_0(x)\,;
\\
W_{{\rm SO}(\rm{even})}(x) \, &= \, 1 + \frac{\sin 2\pi x}{2\pi x}\,;
\\
W_{{\rm SO}({\rm odd})}(x) \, &= \,  1- \frac{\sin 2\pi x}{2\pi x}+\boldsymbol{\delta}_0(x)\,,
\end{split}
\end{equation}
where $\boldsymbol{\delta}_0(x)$ is the Dirac distribution at $x=0$. 

\subsubsection{Hilbert spaces} \label{RKHS_Def} An entire function $F: \C \to \C$ is said to be of exponential type if 
$$\tau(F) := \limsup_{|z| \to \infty} \,|z|^{-1}  \log |F(z)|  < \infty.$$
In this case, the number $\tau(F)$ is called the exponential type of $F$. An entire function $F:\C \to \C$ is said to be real entire if its restriction to $\R$ is real-valued. For $\Delta >0$, the classical Paley-Wiener space, denoted here by $\mathcal{H}_{\pi \Delta}$, is the Hilbert space of entire functions $F$ of exponential type at most $\pi \Delta$ with norm 
$$\|F\|_{\mathcal{H}_{\pi \Delta}} := \|F\|_{L^2(\R)} = \left(\int_{\R} |F(x)|^2 \,\dx \right)^{1/2}< \infty.$$ 
This is a reproducing kernel Hilbert space, i.e.~a Hilbert space in which the evaluation functionals $\Psi_w:\mathcal{H}_{\pi \Delta} \to \C$ given by $\Psi_w(F) = F(w)$, for any fixed $w \in \C$, are continuous. The Paley-Wiener theorem establishes that $F \in \mathcal{H}_{\pi \Delta}$ if and only if $F \in L^2(\R) \cap C(\R)$ and ${\rm supp}(\widehat{F}) \subset [-\tfrac{\Delta}{2}, \tfrac{\Delta}{2}]$. Throughout the paper we adopt the normalization 
\[
\widehat{F}(y) = \int_{\R} e^{-2 \pi i x y} \,F(x) \,\dx \ \ \ {\rm and} \ \ \ \widecheck{F}(y) = \int_{\R} e^{2 \pi i x y} \,F(x) \,\dx
\] 
for the Fourier transform and its inverse, respectively.

\smallskip

Associated to each of our five symmetry groups $G \in \{{\rm U, Sp, O, SO}({\rm even}), {\rm SO}(\rm{odd})\}$, and each parameter $\Delta >0$, we consider the normed vector space $\mathcal{H}_{G, \pi \Delta}$ of entire functions $F$ of exponential type at most $\pi \Delta$ with norm
$$\|F\|_{\mathcal{H}_{G, \pi \Delta}}  := \|F\|_{L^2(\R, W_G)} = \left(\int_{\R} |F(x)|^2 \,W_G(x)\,\dx \right)^{1/2}< \infty.$$ 
Note that $\mathcal{H}_{{\rm U}, \pi \Delta} = \mathcal{H}_{\pi \Delta}$. We verify in Section \ref{Sec2}, via the uncertainty principle for the Fourier transform, that the spaces $\mathcal{H}_{G, \pi\Delta}$ and $\mathcal{H}_{\pi \Delta}$ are the same (as sets) and have equivalent norms. In particular, this implies that $\mathcal{H}_{G, \pi\Delta}$ is also a reproducing kernel Hilbert space. By the Riesz representation theorem, there is map $K_{G, \pi\Delta}: \C \times \C \to \C$ (called the reproducing kernel) such that, for each $w \in \C$, the function $z \mapsto K_{G,\pi \Delta}(w, z)$ belongs to $\mathcal{H}_{G, \pi\Delta}$ and the evaluation functional at $w$ is given by the inner product with $K_{G, \pi\Delta}(w, \cdot)$, that is
\begin{equation}\label{20210504_11:59}
F(w) = \langle F , K_{G, \pi\Delta}(w, \cdot)\rangle_{\mathcal{H}_{G, \pi\Delta}} = \int_{\R} F(x) \, \overline{K_{G, \pi\Delta}(w, x)} \,W_G(x)\,\dx
\end{equation}
for each $F \in \mathcal{H}_{G, \pi\Delta}$. Directly from the definition \eqref{20210504_11:59}, with $F = K_{G, \pi\Delta}(w, \cdot)$, it follows that 
\begin{equation}\label{20210512_09:33}
K_{G, \pi\Delta}(w,w) = \|K_{G, \pi\Delta}(w, \cdot)\|^2_{L^2(\R, W_G)}\geq 0
\end{equation}
for each $w \in \C$. In \S \ref{Sub5_Prelim} we verify that the inequality on the right-hand side of \eqref{20210512_09:33} is actually strict.

\subsubsection{The average order of vanishing of $L$-functions at low-lying heights} We are now in position to state our general result, which gives an upper bound for the average order of vanishing of $L$-functions in a family in terms of special values of the associated reproducing kernel.

\begin{theorem}
\label{Thm2_Non-vanishing} Let $\{L(s,f), \ f \in \mathcal{F}\}$ be a family of $L$-functions with an associated symmetry group $G \in \{{\rm U, Sp, O, SO}({\rm even}), {\rm SO}(\rm{odd})\}$, and assume that GRH holds for $L$-functions in this family. In addition, suppose that estimate \eqref{20210430_11:41} holds for even Schwartz functions $\phi:\R \to \R$ with ${\rm supp} (\widehat{\phi}) \subset (-\Delta, \Delta)$, for a fixed $\Delta >0$. Then, letting $K = K_{G, \pi\Delta}$ be the reproducing kernel of the associated Hilbert space $\mathcal{H}_{G, \pi\Delta}$, the following estimates hold:
\begin{itemize}
\item[(i)] \textup{(Average order of vanishing at the central point)} We have 
\begin{equation}\label{20210512_09:59}
\limsup_{Q \to\infty} \,\frac{1}{|\mathcal{F}(Q)|} \sum_{f \in \mathcal{F}(Q)} \ \underset{s=\frac{1}{2}}{\textup{ord}} \ L\left(s\,,\, f\right) \leq \frac{1}{K(0,0)}.
\end{equation}
\item[(ii)] \textup{(Average order of vanishing at a low-lying height)} For any $t >0$, we have
\begin{equation}\label{20210503_09:43}
\limsup_{Q \to\infty} \,\frac{1}{|\mathcal{F}(Q)|} \sum_{f \in \mathcal{F}(Q)} \ \underset{s=\frac{1}{2}}{\textup{ord}} \ L\left(s+ \frac{2 \pi i t }{\log c_f}\,,\, f\right) \leq \frac{1}{K(t,t) + |K(t,-t)|}.
\end{equation}
\end{itemize}
\end{theorem}

\smallskip

The main novelty in this theorem is part (ii), in which we use the duality condition \eqref{assumption} and the observation \eqref{duality} to substantially improve the estimate for the average order of vanishing at a low-lying height, when compared to the central point (essentially by a factor of 2 as $t \to 0^+$). In \S \ref{SSec:Examples}, we further discuss the reach of this theorem using specific examples and known 1-level density results in the literature. Bounding the order of vanishing at the central point, part (i) above, is a well-studied problem; the numerical upper bound on the right-hand side of \eqref{20210512_09:59} has been obtained by Iwaniec, Luo and Sarnak \cite[Appendix A]{ILS} when $\Delta =2$, and their methods were later extended by Freeman and Miller \cite{F, FM} with explicit computations for small values of $\Delta$. We revisit part (i) of the theorem here for the reader's convenience (to have all such bounds collected in one place), and also because our methods bring a slightly different perspective. As we shall see, there are two extremal problems in Fourier analysis connected to these bounds: for part (i) it is what we call the one-delta extremal problem, whereas for part (ii) it is a more general situation that we call the two-delta extremal problem. Iwaniec, Luo and Sarnak \cite[Appendix A]{ILS}, and later Freeman and Miller \cite{F, FM}, find the extremal functions for the one-delta problem by means of Fredholm theory. Here we provide a unified approach via the framework of reproducing kernel Hilbert spaces, where both extremal problems have an elegant solution in terms of the underlying reproducing kernels (see \textsection \ref{Sec2_proofs}). The hard part then becomes finding explicit representations for these reproducing kernels. We remark that the conceptual bounds on the right-hand sides of \eqref{20210512_09:59} and \eqref{20210503_09:43} are the best possible with such methods.

\subsubsection{Reproducing kernels} In order to be able to fully appreciate the power of the unifying framework of Theorem \ref{Thm2_Non-vanishing}, we need to understand the terms on the right-hand sides of \eqref{20210512_09:59} and \eqref{20210503_09:43}. One of the purposes of this paper, from the analysis point of view, is to discuss the explicit constructions of the reproducing kernels $K_{G, \pi\Delta}$ for the five symmetry groups $G$ and a free parameter $\Delta >0$. These are results of independent interest and may be useful in other situations, and for simplicity we structure them in one theorem per symmetry group. In order to explain the difficulties of this task, we start by recording here the (distributional) Fourier transforms of our five densities in \eqref{densities}:
\begin{align}\label{20210506_11:40_1}
\begin{split}
\widehat{W}_{\rm U}(y) \, &= \, \boldsymbol{\delta}_0(y)\,;
\\
\widehat{W}_{{\rm Sp}}(y) \, &= \, \boldsymbol{\delta}_0(y) - \tfrac{1}{2}\, {\bf 1}_{[-1,1]}(y)\,;
\\
\widehat{W}_{\rm O}(y) \, &= \, \boldsymbol{\delta}_0(y) + \tfrac{1}{2}\,;
\\
\widehat{W}_{\rm SO(even)}(y) \, &= \, \boldsymbol{\delta}_0(y) +\tfrac{1}{2} \,{\bf 1}_{[-1,1]}(y)\,;
\\
\widehat{W}_{\rm SO(odd)}(y) \, &= \, \boldsymbol{\delta}_0(y)- \tfrac{1}{2} \,{\bf 1}_{[-1,1]}(y)+1.
\end{split}
\end{align}
Note that, for each $\phi \in L^1(\R) \cap C(\R)$ with $\widehat{\phi}$ of compact support, by Plancherel's theorem, one has
\begin{equation}\label{20210506_11:42}
\int_{\R} \phi(x) \,W_{G}(x)\,\dx = \int_{\R} \widehat{\phi}(y) \,\widehat{W_{G}}(y)\,\dy.
\end{equation}
For the unitary symmetry, since the underlying Hilbert space $\mathcal{H}_{{\rm U}, \pi \Delta}$ is simply the classical Paley-Wiener space $\mathcal{H}_{\pi \Delta}$, its reproducing kernel is well-known. We recall it here for completeness.

\begin{theorem}[(Paley-Wiener) Reproducing kernel: unitary symmetry]   \label{Thm3_PW}
For any $\Delta >0$, we have
\begin{equation*}
K_{{\rm U},\pi\Delta}(w,z) 
= \dfrac{\sin\pi\Delta(z-\overline{w})}{\pi(z-\overline{w})}.
\end{equation*}
\end{theorem}

For the orthogonal symmetry we provide a complete characterization of the reproducing kernel.

\begin{theorem}[Reproducing kernel: orthogonal symmetry]\label{Rep_Kernel_O} 
For any $\Delta >0$, we have
\begin{align*}
K_{{\rm O},\pi\Delta}(w,z) = \dfrac{\sin\pi\Delta(z-\overline{w})}{\pi(z-\overline{w})} - \frac{1}{(2 + \Delta)} \left(\dfrac{\sin\pi\Delta z}{\pi z}\right) \left( \dfrac{\sin\pi\Delta \overline{w} }{\pi \overline{w}}\right).
\end{align*}
\end{theorem}
The remaining cases $G \in \{{\rm Sp, SO(even), SO(odd)}\}$, especially in the regime $\Delta > 1$, are technically more involved (as it is clear from statements of the forthcoming theorems). The difficulty essentially lies in the fact that one needs to work beyond the discontinuity of the characteristic function ${\bf 1}_{[-1,1]}$ in the Fourier transforms \eqref{20210506_11:40_1}.  In \textsection \ref{Sec_Rep_Kernels} we discuss a method to explicitly construct the reproducing kernels, by an application of the Fredholm alternative. One is naturally led to a problem of inverting a Fredholm operator, which can be done by solving certain differential equations (that grow in order with the parameter $\Delta$) and then solving a linear system of equations to figure out the appropriate constants. We run such a method to obtain the explicit forms when $0 < \Delta \leq 2$, since this covers the vast majority of existing examples in the literature, but the philosophy could be applied in general for any particular choice of $\Delta>2$, with a higher computational cost. A similar step in spirit, with related inherent computational difficulties, appears in the approach of Freeman and Miller \cite{F, FM} for bounds for the order of vanishing at the central point (there, the constants that need to be computed are complex numbers, while here they are in fact functions of another variable).

\newpage

\begin{theorem}[Reproducing kernel: even orthogonal symmetry]\label{Rep_Kernel_SO(even)} \hfill

\noindent {\rm (i)} For $0 <\Delta \leq 1$, we have
\begin{align*}
\begin{split}
K_{{\rm SO(even)},\pi\Delta}(w,z) = \dfrac{\sin\pi\Delta(z-\overline{w})}{\pi(z-\overline{w})} - \frac{1}{(2 + \Delta)} \left(\dfrac{\sin\pi\Delta z}{\pi z}\right) \left( \dfrac{\sin\pi\Delta \overline{w} }{\pi \overline{w}}\right).
\end{split}
\end{align*}

\smallskip

\noindent {\rm (ii)} For $1 < \Delta \leq 2$, define the constants 
\begin{align*}
\tau &:= e^{(2-\Delta)i/4}  + i e^{\Delta i/4}\,;\\
a &:= e^{\Delta i /4} + i e^{(2-\Delta)i/4} - ie^{ \Delta i/4} +  \left( \tfrac{2 - \Delta}{4}\right)\tau\,;\\
b& :=  e^{\Delta i/4} + ie^{(2-\Delta) i/4} - e^{(2- \Delta) i/4} + \left( \tfrac{2 - \Delta}{4}\right)\tau\,,
\end{align*}
and functions of a complex variable $w$,
\begin{align*}
C(w)& := \frac{-16\pi^2w^2 -4 \pi i w\,e^{2\pi iw}}{1 - 16 \pi^2w^2} \ \  ; \ \ \  F(w) := \frac{2\cos(\pi(2-\Delta)w) - 8 \pi w \sin(\pi \Delta w)}{1 - 16 \pi^2 w^2}\,;\\
G(w)&:=\frac{2\cos(\pi \Delta w)+ 4 \pi i w\, e^{- \pi \Delta i w} - e^{\pi  (2 -\Delta)i w}  }{1 - 16 \pi^2 w^2}  - \frac{\sin(\pi (2 - \Delta)w)}{2\pi w (1 - 16 \pi^2 w^2) } +  \left( \frac{2 - \Delta}{4}\right)F(w) \,;\\
A(w) & := \frac{\overline{a}\,\overline{G(\overline{w})} - \overline{b}\,G(w)}{\overline{a} b - a\overline{b}}\ \ ; \ \ B(w)  := \frac{b\,G(w) - a\,\overline{G(\overline{w})}}{\overline{a}b - a \overline{b}} \ \ ; \ \ D(w) := \frac12 \big(\tau A(w) + \overline{\tau} B(w) - F(w)\big).
\end{align*}
Then 
\begin{align}\label{20210512_11:40}
\begin{split}
 \!\!\! \! K_{{\rm SO(even)},\pi\Delta}(w,z)&  = \frac{\overline{A(w)} \big(i e^{2\pi i z} - 1\big)\big(e^{-\Delta( \pi i z + i/4)} - e^{-(2-\Delta)(\pi i z + i/4)}\big)}{2\pi i z + i/2} \\
&   \qquad  +\frac{\overline{B(w)}\big(\! -i e^{2\pi i z} - 1\big)\big(e^{-\Delta( \pi i z - i/4)} - e^{-(2-\Delta)(\pi i z - i/4)}\big)}{2\pi i z - i/2} \\
& \qquad + \frac{C(\overline{w})\big( e^{\pi \Delta i (z - \overline{w})} - e^{\pi (2-\Delta) i (z - \overline{w})}\big)   + \overline{C(w)}\big(-e^{-\pi \Delta  i (z-\overline{w})} + e^{-\pi (2 - \Delta) i (z-\overline{w})}\big)}{2\pi i (z-\overline{w})}\\
&  \qquad  + \frac{\overline{D(w)} \, \sin\big(\pi (2-\Delta) z\big)}{\pi z} + \frac{\sin\big(\pi (2-\Delta) (z - \overline{w})\big)}{\pi(z-\overline{w})}.
\end{split}
\end{align}
At the points $w =  \pm1/4\pi$, the formula \eqref{20210512_11:40} should be interpreted as the appropriate limit since $w \mapsto K_{{\rm SO(even)},\pi\Delta}(\overline{w},z)$ is entire.
\end{theorem}
\noindent {\sc Remark:} When $\Delta =2$, note that the last line of \eqref{20210512_11:40} disappears. In this case, the auxiliary functions $D$ and $F$, and the constant $\tau$, play no role. A similar simplification happens when $\Delta =2$ in the next two theorems, for symplectic and odd orthogonal symmetries.

\begin{theorem}[Reproducing kernel: symplectic symmetry]\label{Thm_symplectic} \hfill

\noindent {\rm (i)} For $0 <\Delta \leq 1$ we have
\begin{align*}
K_{{\rm Sp},\pi\Delta}(w,z)= \dfrac{\sin\pi\Delta(z-\overline{w})}{\pi(z-\overline{w})} + \frac{1}{(2 - \Delta)} \left(\dfrac{\sin\pi\Delta z}{\pi z}\right) \left( \dfrac{\sin\pi\Delta \overline{w} }{\pi \overline{w}}\right).
\end{align*} 

\smallskip

\noindent {\rm (ii)} For $1 < \Delta \leq 2$, define the constants
\begin{align*}
\tau &:= e^{(2-\Delta)i/4}  - i e^{\Delta i/4}\,;\\
a &:= e^{\Delta i /4} - i e^{(2-\Delta)i/4} + ie^{ \Delta i/4} -\left( \tfrac{2 - \Delta}{4}\right)\tau\,;\\
b& :=  e^{\Delta i/4} - ie^{(2-\Delta) i/4} - e^{(2- \Delta) i/4} -\left( \tfrac{2 - \Delta}{4}\right)\tau\,,
\end{align*}
and functions of a complex variable $w$,
\begin{align*}
C(w)& := \frac{-16\pi^2w^2 +4 \pi i w\,e^{2\pi iw}}{1 - 16 \pi^2w^2} \ \  ; \ \ \  F(w) := \frac{2\cos(\pi(2-\Delta)w) + 8 \pi w \sin(\pi \Delta w)}{1 - 16 \pi^2 w^2}\,;\\
G(w)&:=\frac{2\cos(\pi \Delta w)- 4 \pi i w \,e^{- \pi \Delta i w} - e^{\pi  (2 -\Delta)i w}  }{1 - 16 \pi^2 w^2}  + \frac{\sin(\pi (2 - \Delta)w)}{2\pi w (1 - 16 \pi^2 w^2) } -  \left( \frac{2 - \Delta}{4}\right)F(w) \,;\\
A(w) & := \frac{\overline{a}\,\overline{G(\overline{w})} - \overline{b}\,G(w)}{\overline{a} b - a\overline{b}}\ \ ; \ \ B(w)  := \frac{b\,G(w) - a\,\overline{G(\overline{w})}}{\overline{a}b - a \overline{b}} \ \ ; \ \ D(w) := \frac12 \big(\tau A(w) + \overline{\tau} B(w) - F(w)\big).
\end{align*}
Then 
\begin{align}\label{20210825_10:58am}
\begin{split}
 \!\!\! \! K_{{\rm Sp},\pi\Delta}(w,z)&  = \frac{\overline{A(w)} \big(-i e^{2\pi i z} - 1\big)\big(e^{-\Delta( \pi i z + i/4)} - e^{-(2-\Delta)(\pi i z + i/4)}\big)}{2\pi i z + i/2} \\
&   \qquad  +\frac{\overline{B(w)}\big(i e^{2\pi i z} - 1\big)\big(e^{-\Delta( \pi i z - i/4)} - e^{-(2-\Delta)(\pi i z - i/4)}\big)}{2\pi i z - i/2} \\
& \qquad + \frac{C(\overline{w})\big( e^{\pi \Delta i (z - \overline{w})} - e^{\pi (2-\Delta) i (z - \overline{w})}\big)   + \overline{C(w)}\big(-e^{-\pi \Delta  i (z-\overline{w})} + e^{-\pi (2 - \Delta) i (z-\overline{w})}\big)}{2\pi i (z-\overline{w})}\\
&  \qquad  + \frac{\overline{D(w)} \, \sin\big(\pi (2-\Delta) z\big)}{\pi z} + \frac{\sin\big(\pi (2-\Delta) (z - \overline{w})\big)}{\pi(z-\overline{w})}.
\end{split}
\end{align}
At the points $w =  \pm1/4\pi$, the formula \eqref{20210825_10:58am} should be interpreted as the appropriate limit since $w \mapsto K_{{\rm Sp},\pi\Delta}(\overline{w},z)$ is entire.
\end{theorem}

\begin{theorem}[Reproducing kernel: odd orthogonal symmetry] \label{Thm_odd_orthogonal}
\hfill

\noindent {\rm (i)} For $0 <\Delta \leq 1$, we have
\begin{align*}
\begin{split}
K_{{\rm SO(odd)},\pi\Delta}(w,z) = \dfrac{\sin\pi\Delta(z-\overline{w})}{\pi(z-\overline{w})} - \frac{1}{(2 + \Delta)} \left(\dfrac{\sin\pi\Delta z}{\pi z}\right) \left( \dfrac{\sin\pi\Delta \overline{w} }{\pi \overline{w}}\right).
\end{split}
\end{align*}

\smallskip

\noindent {\rm (ii)} For any $\Delta >0$, we have
\begin{align*}
K_{{\rm SO(odd)},\pi\Delta}(w,z) = K_{{\rm Sp},\pi\Delta}(w,z) - \frac{K_{{\rm Sp},\pi\Delta}(w,0)}{1 + K_{{\rm Sp},\pi\Delta}(0,0)}K_{{\rm Sp},\pi\Delta}(0,z).
\end{align*}
\end{theorem}

For completeness, let us record here the relevant input for the right-hand side of \eqref{20210512_09:59}.

\begin{corollary}[Reproducing kernels: special values at the origin] \label{Cor_8_00}
\hfill

\noindent {\rm (i)} \textup{(Unitary and orthogonal symmetry)} For any $\Delta>0$, we have
\[
K_{{\rm U},\pi\Delta}(0,0)= \Delta \quad \text{ and } \quad K_{{\rm O},\pi\Delta}(0,0)= \frac{2\Delta}{2 + \Delta}.
\]
\noindent {\rm (ii)} \textup{(Even orthogonal symmetry)} We have
\[
K_{{\rm SO(even)},\pi\Delta}(0,0) = \left\{ \begin{array}{lcl}
 2\Delta / (2 + \Delta), &\mbox{ if $0<\Delta \le 1$,}
\\[\bigskipamount]
   2 - \dfrac{ 4 \cos\!\big( \frac{\Delta-1}{2}  \big) }{4  + 4 \sin\!\big( \frac{\Delta-1}{2} \big) - \Delta  \cos\!\big( \frac{\Delta-1}{2} \big)} , &\mbox{ if $1<\Delta \le 2$.}
       \end{array} \right.
\]

\noindent {\rm (iii)} \textup{(Symplectic symmetry)} We have
\[
K_{{\rm Sp},\pi\Delta}(0,0) = \left\{ \begin{array}{lcl}
2\Delta / (2 - \Delta), &\mbox{ if $0<\Delta \le 1$,} 
\\[\bigskipamount]
  \dfrac{ 4 \cos\!\big( \frac{\Delta-1}{2}  \big) }{4 - 4 \sin\!\big( \frac{\Delta-1}{2} \big) - \big(4 - \Delta \big) \cos\!\big( \frac{\Delta-1}{2} \big)} -  2, &\mbox{ if $1<\Delta \le 2$.}
       \end{array} \right.
\]

\noindent {\rm (iv)} \textup{(Odd orthogonal symmetry)} We have
\[
K_{{\rm SO(odd)},\pi\Delta}(0,0) = \left\{ \begin{array}{lcl}
2\Delta/(2 + \Delta), &\mbox{ if $0<\Delta \le 1$,} 
\\[\bigskipamount]
  2 + \dfrac{ 4 \cos\!\big( \frac{\Delta-1}{2}  \big) }{4 -4 \sin\!\big( \frac{\Delta-1}{2} \big) - \big(8 - \Delta \big) \cos\!\big( \frac{\Delta-1}{2} \big)}, &\mbox{ if $1<\Delta \le 2$.}
       \end{array} \right.
\]
\end{corollary}
The corresponding inputs from Theorems \ref{Thm3_PW} -- \ref{Thm_odd_orthogonal} needed for the right-hand side of \eqref{20210503_09:43} are computable for any fixed $t >0$ (but the expressions do not dramatically simplify to justify stating them in a new corollary). We now illustrate the reach of our framework by highlighting a few examples and plotting the corresponding lower bound for the proportion of non-vanishing at low-lying heights implied by \eqref{20210503_09:43}.

\subsection{Some examples and plots of non-vanishing}\label{SSec:Examples} There are many examples of families of $L$-functions in the literature which are amenable to our framework. A non-exhaustive list includes, for instance, the works 
\cite{AM, AB,BZ,BF,CK,CS,DPR,Du,DM,FI,GZ,G,HB,HM,HR,ILS,KS1,MP,OS,RR,R,Ru,SST,Y}.
In order to keep the exposition short, the reader is invited to consult the corresponding excerpts from these works for the precise definitions and technicalities, and we only briefly comment on a few of these examples below. Recall that the relevant input for our purposes is that estimate \eqref{20210430_11:41} holds for even Schwartz functions $\phi:\R \to \R$ with ${\rm supp} (\widehat{\phi}) \subset (-\Delta, \Delta)$ for some $\Delta >0$. The lower bound for the proportion of non-vanishing at a low-lying height $t>0$ implied by \eqref{20210503_09:43} is $\mc{P} = \mc{P}_{G, \pi \Delta}$ given by
\begin{align*}
\liminf_{Q \to\infty} \ \frac{1}{|\mathcal{F}(Q)|} \cdot \#\Big\{ L\big(\tfrac{1}{2}+\tfrac{2\pi i t}{\log c_f}\big) \ne 0  : f \in \mathcal{F}(Q) \Big\} \, \ge \, \mc{P}(t)  \, :=  \, 1 \, - \,\frac{1}{K(t,t) + |K(t,-t)|},
\end{align*}
where $K = K_{G, \pi\Delta}$.

\subsubsection{Unitary examples} The case of primitive Dirichlet $L$-functions modulo a prime $q$, presented in Theorem \ref{Dirichlet1}, relies on the work of Hughes and Rudnick \cite[Theorem 3.1]{HR}, with unitary symmetry ($G= {\rm U}$) and $\Delta =2$. The corresponding plot of the lower bound for the proportion of non-vanishing implied by Theorem \ref{Dirichlet1}, which follows from \eqref{20210503_09:43}, is given in Figure \ref{figure1}. In \cite[Theorem 1]{DPR}, Drappeau, Pratt, and Radziwi\l\l \, show that the larger family of primitive Dirichlet $L$-functions modulo $q$ with an additional smoothed average on the parameter $q$ in an interval ($q$ running over all integers in the interval) is a unitary family of $L$-functions  with extended support corresponding to $\Delta = 2 + \frac{50}{1093}$. This result can be used in conjunction with Theorems \ref{Thm2_Non-vanishing} and \ref{Thm3_PW} to give improved estimates for the proportion of non-vanishing of $L$-functions in this family at low-lying heights on the critical line (we discuss how to deal with a smooth averaging in a remark below).

\subsubsection{\textup{(}Even and odd\textup{)}~orthogonal examples}\label{mf} As examples of orthogonal families of automorphic $L$-functions, we highlight the ones described by Iwaniec, Luo, and Sarnak \cite[Theorems 1.1--1.3]{ILS}. Following the notation in \cite{ILS}, for $k$ even and $N$ square-free, let $H_k^\star(N)$ denote the set of holomorphic cusp forms $f$ of weight $k$ which are newforms of level $N$, and let $L(s,f)$ be the associated $L$-function. Using the sign of the functional equation, this set naturally decomposes into two subsets, $H_k^+(N)$ and $H_k^-(N)$, of forms $f$ corresponding to $\varepsilon_f = +1$ and $\varepsilon_f = - 1$, respectively. As the level $N \to \infty$ over square-free integers, \cite[Theorem 1.1]{ILS} gives that the families $H_k^\star(N)$, $H_k^+(N)$, and $H_k^-(N)$ have symmetry types $G = {\rm O}$, $G = {\rm SO}({\rm even})$, and $G =  {\rm SO}(\rm{odd})$, respectively, with support $\Delta=2$.  If one considers the same families and lets $kN \to \infty$, then \cite[Theorem 1.2]{ILS} establishes the same symmetry types but with $\Delta=1$. Averaging over the weight $k$, they define the larger families
\[
\mc{M}^\star(K,N)=\bigcup_{k\le K}H_k^\star(N), \quad \mc{M}^+(K,N)=\bigcup_{k\le K}H_k^+(N), \quad \text{and} \quad \mc{M}^-(K,N)=\bigcup_{k\le K}H_k^-(N),
\]
and in \cite[Theorem 1.3]{ILS} they establish that as $KN \to \infty$ these families of $L$-functions have symmetry types $G = {\rm O}$, $G = {\rm SO}({\rm even})$, and $G =  {\rm SO}(\rm{odd})$, respectively, with the larger support $\Delta=2$. Plots for the proportion of non-vanishing implied by \eqref{20210503_09:43} for \textup{(}even and odd\textup{)}~orthogonal families of $L$-functions when $\Delta=1$ and $\Delta=2$ are given in Figures \ref{figure_O}, \ref{figure_SO(even)}, and \ref{figure_SO(odd)}. In particular, Figure \ref{figure_SO(odd)} illustrates a nice example of so-called zero repulsion. For every $f \in H_k^-(N)$, the sign of functional equation implies that $L(\frac{1}{2},f)=0$. However, when $\Delta=2$ we are able to show that a large proportion of this family cannot vanish in certain ranges near central point (the zeros are repelled from $s=\frac{1}{2}$). For instance, Figure \ref{figure_SO(odd)} shows that, for {\it every} $t \in [\frac{3}{10},\frac{4}{10}]$, at least 70\% of the $f \in H_k^-(N)$ have $L\big(\tfrac{1}{2}+\tfrac{2\pi i t}{\log k^2N},f\big) \ne 0$ as $N \to \infty$ over square-free integers. 

\begin{figure} 
\includegraphics[scale=0.7]{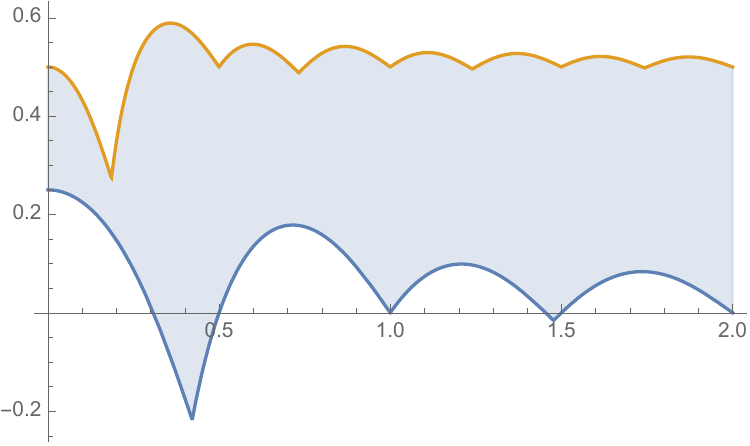}
\caption{Plots of $\mc{P}(t)$, the lower bound for the proportion of non-vanishing, when $G = {\rm O}$ and $\Delta \in \{1,2\}$ (blue and orange, respectively), for small $t>0$. When $\Delta =2$ the maximum value is $0.5892\ldots$ attained at $t = 0.3575\ldots$, and the limit as $t \to \infty$ is $\frac12$. }
\label{figure_O}
\end{figure}

\begin{figure} 
\includegraphics[scale=0.7]{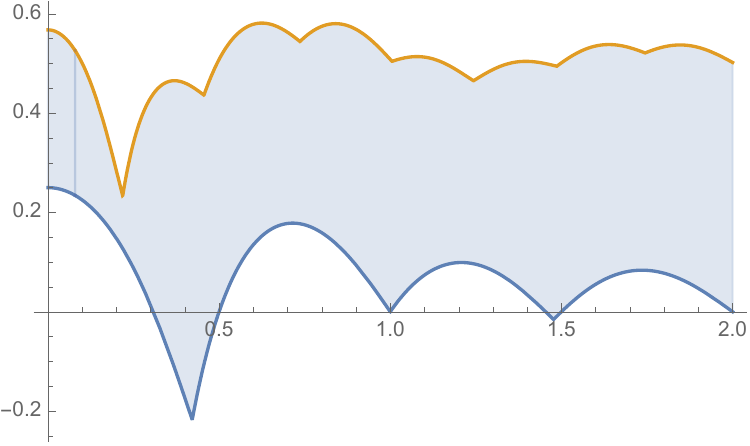}
\caption{Plots of $\mc{P}(t)$, the lower bound for the proportion of non-vanishing, when $G = {\rm SO}({\rm even})$ and $\Delta \in \{1,2\}$ (blue and orange, respectively), for small $t>0$. When $\Delta =2$ the maximum value is $0.5814\ldots$ attained at $t = 0.6247\ldots$, and the limit as $t \to \infty$ is $\frac12$.}
\label{figure_SO(even)}
\end{figure}

\subsubsection{Symplectic examples} We now describe a few examples of symplectic families of $L$-functions. Let $D$ be square-free, $D > 3$, and $D \equiv 3 \,({\rm mod} \,4)$. In \cite{FI}, Fouvry and Iwaniec consider the family of $L$-functions $L(s, \Psi)$, where $\Psi$  
runs over the characters of the ideal class group $\mathcal{C}\ell(K)$ of the imaginary quadratic field $K=\mathbb{Q}(\sqrt{-D})$. In \cite[Theorem 1.1]{FI}, they verify \eqref{20210430_11:41} holds for this family with $G = {\rm Sp}$ for $\Delta =1$ and, in \cite[Theorem 1.2]{FI}, with an additional restricted average over $D$, they are able to prove that this larger family is symplectic with extended support $\Delta = 4/3$.\footnote{They also show that extending the support is naturally connecting to counting primes $p$ of the form $4p=m^2+Dn^2$.} 
In \cite[Theorem 1.5]{ILS}, another example of a symplectic family of $L$-functions is given. Here they consider the family of symmetric square $L$-functions, $L(s,\mathrm{sym}^2f)$, for forms $f\in \mc{M}^\star(K,N)$ as described above. As $K \to \infty$, it is shown that this family satisfies \eqref{20210430_11:41} with $G = {\rm Sp}$ and support $\Delta = 3/2$. Finally, \"{O}zl\"{u}k and Snyder \cite{OS} showed that the family of quadratic Dirichlet $L$-functions (averaging over fundamental discriminants $d$) satisfies \eqref{20210430_11:41} with $G = {\rm Sp}$ and support $\Delta = 2$. They work with a weighted average of $d$, but the formulation of their result in the recent work of Gao \cite[Theorem 2.1 and Equation (2.2)]{Gao} agrees with our setup. Plots of the lower bounds for the proportion of non-vanishing at low-lying heights for these symplectic families of $L$-functions are given in Figure \ref{figure_Sp}. We draw the attention to the fact that, when $\Delta =2$, for every $t$ with $0<t<\frac{1}{16}$, one has $L\big(\tfrac{1}{2}+\tfrac{2\pi i t}{\log c_f},f\big) \ne 0$ for at least $94\%$ of the $L$-functions in the family.

\begin{figure} 
\includegraphics[scale=0.7]{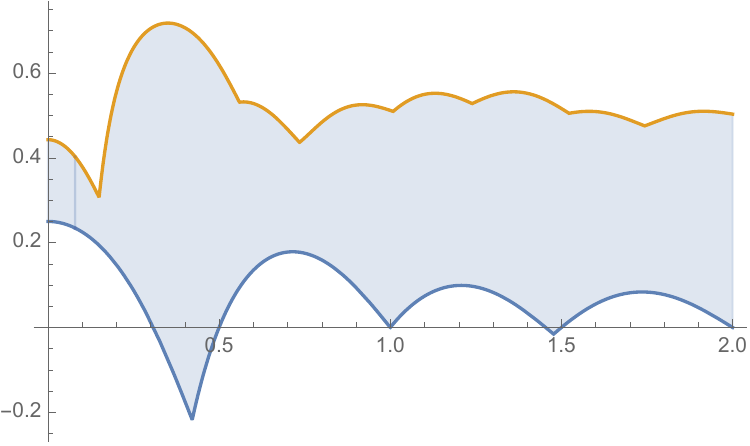}
\caption{Plots of $\mc{P}(t)$, the lower bound for the proportion of non-vanishing, when $G = {\rm SO}({\rm odd})$ and $\Delta \in \{1,2\}$ (blue and orange, respectively), for small $t>0$. When $\Delta =2$ the maximum value is $0.7175\ldots$ attained at $t = 0.3505\ldots$, and the limit as $t \to \infty$ is $\frac12$.}  
\label{figure_SO(odd)}
\end{figure}

\begin{figure} 
\includegraphics[scale=0.7]{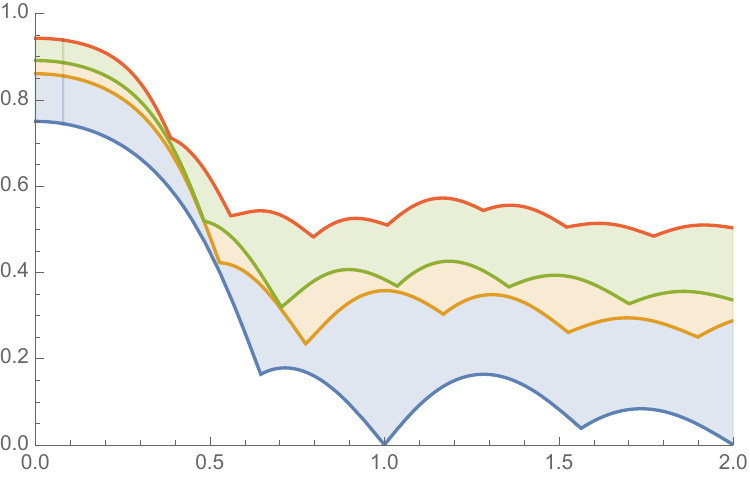}
\caption{Plots of $\mc{P}(t)$, the lower bound for the proportion of non-vanishing, when $G = {\rm Sp}$ and $\Delta \in \{1,\frac{4}{3}, \frac{3}{2}, 2\}$ (blue, orange, green, and red, respectively), for small $t>0$. The maximum values are attained when $t \to 0^+$ and can be explicitly computed using Corollary \ref{Cor_8_00}. In particular, $\mc{P}(0^+)$ is equal to $\frac{3}{4}$ (for $\Delta =1$), $0.8604\ldots$ (for $\Delta =\frac{4}{3}$), $0.8910\ldots$ (for $\Delta =\frac{3}{2}$), and $0.9427\ldots$ (for $\Delta =2$). From Theorem \ref{Thm_symplectic}, the limit of $\mc{P}(t)$ as $t \to \infty$ is $1 - \frac{1}{\Delta}$ in each case.}
\label{figure_Sp}
\end{figure}

\smallskip

\noindent {\sc Remark:} There are instances in the literature in which a small variation of \eqref{20210430_11:41} is proved, namely, when the average on the left-hand side of \eqref{20210430_11:41} is replaced by a weighted/smoothed average. A typical example would be when our family $\mc{F}$ decomposes as $\displaystyle \mc{F} = \bigcup_q \mc{F}(q)$ and one smoothly averages over the parameter $q$ in an interval, say $Q \leq q \leq 2Q$, in order to prove an identity of the form 
\begin{equation*}
\lim_{Q \to\infty} \, \frac{ \displaystyle \sum_q \Psi(q/Q)\sum_{f\in \mathcal{F}(q)} \sum_{\gamma_f} \phi\!\left( \gamma_f \frac{\log c_f}{2\pi}\right)}{\displaystyle \sum_q \Psi(q/Q) \cdot  |\mc{F}(q)|   }= \int_\mathbb{R} \phi(x) \, W_G(x) \, \d x,
\end{equation*}
where $\Psi$ is a smooth function supported in the interval $[1,2]$. Assuming that the weight $\Psi$ is non-negative, 
the proof of our Theorem \ref{Thm2_Non-vanishing} below carries through {\it ipsis litteris}, yielding, for $t>0$,
\begin{equation*}
\limsup_{Q \to\infty} \frac{\displaystyle \sum_q \Psi(q/Q) \sum_{f \in \mathcal{F}(q) } \ \underset{s=\frac{1}{2}}{\textup{ord}} \ L\left(s+ \frac{2 \pi i t }{\log c_f}\,,\, f\right)}{ \displaystyle \sum_q \Psi(q/Q) \cdot  |\mc{F}(q)| } \, \leq \, \frac{1}{K(t,t) + |K(t,-t)|}.
\end{equation*}
This is how the upper bound for the proportion of vanishing (and, consequently, the lower bound for proportion of non-vanishing) should be interpreted in these cases.

 \subsection{Existence of low-lying zeros} \label{Intro_low_lying_zeros}In \cite[Theorem 8.1]{HR}, Hughes and Rudnick consider the problem of establishing the {\it existence} of low-lying zeros in the family of primitive Dirichlet $L$-functions modulo $q$ by relating it to an interesting extremal problem in analysis. They prove that, under GRH, if the support is $[-\Delta,\Delta]$ in the one-level density theorem for Dirichlet $L$-functions modulo a prime $q$, then
\begin{equation*}
\limsup_{q \to \infty} \,\min_{\substack{\gamma_\chi \\ \chi \neq \chi_0}} \left|\frac{\gamma_{\chi} \,\log q}{2\pi}\right|\leq \frac{1}{2\Delta}\,,
\end{equation*}
where the minimum is taken over the ordinates $\gamma_\chi$ of the zeros of $L(s,\chi)$ for all non-principal characters $\chi$ (mod $q$). In other words, there exist low-lying zeros that are at most $\frac{1}{2\Delta}$ times the average spacing. With $\Delta = 2$, they are able to deduce the value $\frac{1}{4}$ for this problem. Their setup will work for any unitary family of $L$-functions, and our goal is to discuss a generalization of their result for the other four symmetry types. We remark that such a generalization was first obtained by Bernard \cite{Bernard}, but our methods here are different (and so is our setup in the cases $G = {\rm O}$ or $G = {\rm SO}({\rm odd})$).

\smallskip

Given a family $\mathcal{F}$, we might aim to seek an upper bound for 
\[
\limsup_{Q\to\infty} \ \min_{ \substack{\gamma_f \\ f \in \mathcal{F}(Q) }} \left| \frac{\gamma_f \log c_f}{2\pi} \right|,
\]
where the minimum is over the ordinates $\gamma_f$ of the zeros of the $L$-functions $L(s,f)$ with $f \in \mathcal{F}(Q)$. Indeed, for even orthogonal or symplectic families, we estimate this quantity. However, for an orthogonal (resp.~odd orthogonal) family, we expect half (resp.~all) of the $L$-functions in the family to trivially vanish at the central point due to the sign of the functional equation. Therefore, for these two symmetry types, we define a slightly modified quantity. As we shall see, in all four cases it becomes more difficult to solve the associated extremal problem in analysis, stated in \eqref{20210907_13:49} below, when the corresponding measure is not the Lebesgue measure (as is the case for unitary symmetry). We do so by establishing a bridge to the powerful theory of de Branges spaces of entire functions \cite{Branges}, and by taking advantage of our explicit representations for the reproducing kernels given in Theorems \ref{Thm3_PW} -- \ref{Thm_odd_orthogonal}.

\smallskip

Before setting up the general problem, we first give an example to illustrate one of the results. Using the notation in \textsection \ref{mf}, we consider the even orthogonal family $H_k^+\!(N)$ of holomorphic newforms of weight $k$, level $N$, and $\varepsilon_f=+1$. 
As the level $N \to \infty$ over square-free integers, \cite[Theorem 1.1]{ILS} implies that we may take $G = {\rm SO}({\rm even})$ and $\Delta=2$  
for this family. Therefore, Theorem \ref{Thm_Db_9} below implies that
\[
\limsup_{ N \to\infty}  \min_{ \substack{\gamma_f \\ f \in H_k^+\!(N) }} \left| \frac{\gamma_f \log k^2N }{2\pi} \right| \le 0.2185\ldots.
\]
In other words, as $N\to \infty$ over square-free integers, we see that there exist $L(s,f)$ with $f \in H_k^+(N)$ whose first low-lying zero is at most $0.2185\ldots$ times the average spacing. The fact that this quantity is less than the corresponding value of $\frac{1}{4}$ in Hughes and Rudnick's result for Dirichlet $L$-functions is not surprising when we recall that the corresponding density functions are $W_{{\rm U}}(x) = \, 1$ and $W_{{\rm SO}(\rm{even})}(x) \, = \, 1 + \frac{\sin 2\pi x}{2\pi x}$, respectively. In both cases the support is $\Delta=2$, so these densities indicate that we expect more zeros in an even orthogonal family near the central point than we do for a unitary family. The analysis of the problem detects this and we are able to obtain a stronger result.

\smallskip

We now set up our general problem, in which we work under the same hypotheses of Theorem \ref{Thm2_Non-vanishing}. When $G = {\rm O}$ or $G = {\rm SO}({\rm odd})$, we are going to assume that there exists a subset $\bigcup_{f \in \mc{F}}\mc{Z}(f)$ of the zeros at the central point such that 
\begin{equation*}
\lim_{Q \to \infty}  \frac{1}{|\mc{F}(Q)|} \sum_{f \in \mc{F}(Q)}  \sum_{\gamma_f \in \mc{Z}(f)} \!\!1\, = \frac{1}{2} \ \ \  ({\rm when} \ G = {\rm O}) 
\end{equation*}
and
\begin{equation*}
\lim_{Q \to \infty}  \frac{1}{|\mc{F}(Q)|} \sum_{f \in \mc{F}(Q)} \sum_{\gamma_f \in \mc{Z}(f)}  \!\! 1\,= 1 \ \ \  ({\rm when} \ G = {\rm SO}({\rm odd})).
\end{equation*}
When $G \in \{{\rm U, Sp, SO}({\rm even})\}$ we may simply regard $\mc{Z}(f) = \emptyset$. Hence, removing these zeros at the central point that `naturally' arise from the functional equation, our target is to bound the quantity
\begin{equation*}
\limsup_{Q \to \infty} \ \min_{\substack{ \gamma_f \notin \mc{Z}(f) \\ f \in \mc{F}(Q)}} \left|\frac{\gamma_{f} \,\log c_f}{2\pi}\right|.
\end{equation*}
These assumptions remove the delta distributions from the original density functions \eqref{densities}, and \eqref{20210430_11:41} becomes
\begin{equation}\label{20210907_11:58}
\lim_{Q \to\infty} \, \frac{1}{|\mathcal{F}(Q)|} \sum_{f\in \mathcal{F}(Q)} \, \sum_{\gamma_f \notin \mc{Z}(f)} \phi\!\left( \gamma_f \frac{\log c_f}{2\pi}\right) = \int_\mathbb{R} \phi(x) \, W_{G^\sharp}(x) \, \d x,
\end{equation}
where we are denoting 
\begin{equation}\label{20210909_23:00}
G^{\sharp} := G \ \ {\rm if}\ \  G \in \{{\rm U, Sp, SO}({\rm even})\} \ \ \ ; \ \ \ {\rm O}^{\sharp} := {\rm U} \ \ \ ;\  \ \  {\rm SO}({\rm odd})^{\sharp} := {\rm Sp}.
\end{equation}
In other words, we have
\begin{equation*}
W_{{\rm U}^{\sharp}}(x) = W_{{\rm O}^{\sharp}}(x) = 1\ \ ;\ \ W_{{\rm Sp}^{\sharp}}(x) = W_{{\rm SO}({\rm odd})^{\sharp}}(x) = 1 - \frac{\sin 2\pi x}{2\pi x} \ \ ;\ \ W_{{\rm SO}({\rm even})^{\sharp}}(x) = 1 + \frac{\sin 2\pi x}{2\pi x}. 
\end{equation*}
Note here, by abuse of notation, that the entity $G^{\sharp}$ no longer corresponds to the symmetry group of the family. We have introduced this notation simply for convenience, so that we can easily refer to the Hilbert spaces and the corresponding reproducing kernels that we have already defined.
\smallskip

We establish the following result, as a consequence of a more general result within the theory in de Branges spaces (Theorem \ref{DB_Thm} below).

\begin{theorem} \label{Thm_Db_9} Under the same hypotheses of Theorem \ref{Thm2_Non-vanishing}, let $G^{\sharp}$ be given by \eqref{20210909_23:00} and $K = K_{G^{\sharp}, \pi\Delta}$ be the reproducing kernel of the Hilbert space $\mathcal{H}_{G^{\sharp}, \pi\Delta}$. Let $\xi_0$ be the smallest positive real zero of the function $x \mapsto {\rm Re} \big((1-ix)K(i,x)\big)$. Then 
\begin{equation*}
\limsup_{Q\to \infty} \ \min_{\substack{ \gamma_f \notin \mc{Z}(f) \\ f \in \mc{F}(Q)}} \left|\frac{\gamma_{f} \,\log c_f}{2\pi}\right| \leq \xi_0.
\end{equation*}
\end{theorem}
\begin{figure} 
\includegraphics[scale=0.7]{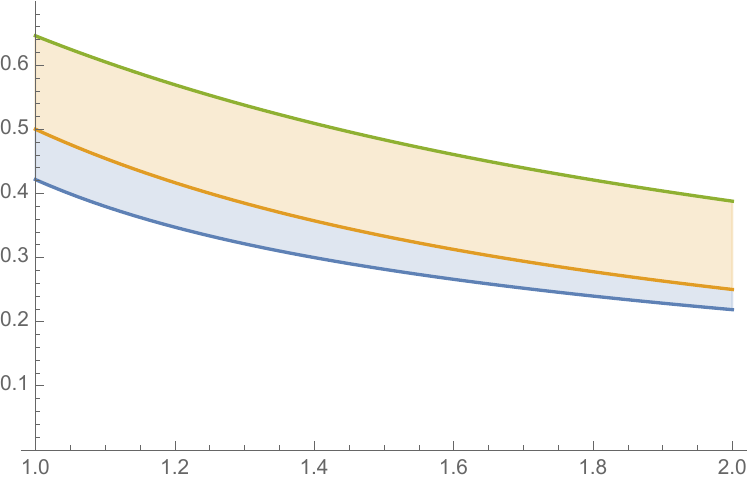}
\caption{Plots of the map $\Delta \mapsto (\xi_0)_{G, \pi\Delta}$ for $1 \leq \Delta \leq 2$ and $G = {\rm SO}({\rm even})$ (in blue), $G \in \{{\rm U, O}\}$ (in orange), and $G \in \{{\rm Sp, SO}({\rm odd})\}$ (in green).}
\label{figure_Rudnick}
\end{figure}
We shall see that this conceptual bound is the best possible with this method. In the cases $G \in \{{\rm U, O}\}$ ($G^{\sharp} = {\rm U}$), from Theorem \ref{Thm3_PW} we have ${\rm Re} \big((1-ix)K(i,x)\big) = \frac{\sinh(\pi \Delta)}{\pi} \cos (\pi \Delta x)$ and we immediately get $\xi_0 = 1/(2\Delta)$, recovering the original result of Hughes and Rudnick \cite{HR} via a different approach. The power of this framework becomes evident in the other cases, in which we use the explicit representations for the reproducing kernels in Theorems \ref{Rep_Kernel_SO(even)} and \ref{Thm_symplectic}. The table below collects some particular values of the upper bound $\xi_0 = (\xi_0)_{G, \pi\Delta}$ for the height of the first low-lying zero (see Figure \ref{figure_Rudnick} for the plots):
\medskip
\begin{center}
\begin{tabular}{|c|c|c|c|c|}
\hline
& \ $\Delta =1$ \ & $\Delta = 4/3$ & $\Delta = 3/2$ & \ $\Delta =2$\  \\
\hline
$G = G^{\sharp}  = {\rm SO}({\rm even})$ &0.4215\ldots&0.3136\ldots&0.2815\ldots&0.2185\ldots \\
\hline
$G \in \{{\rm U, O}\} \ (G^{\sharp} = {\rm U})$ &1/2 &3/8 &1/3 &1/4\\
\hline
$G \in \{{\rm Sp, SO}({\rm odd})\} \ (G^{\sharp} = {\rm Sp})$ &0.6457\ldots &0.5277\ldots&0.4836\ldots&0.3877\ldots \\
\hline
\end{tabular}
\end{center}

\subsection{Notation} Throughout the text we denote by ${\bf 1}$ (resp. ${\bf 0}$) the constant function equal to $1$ (resp. $0$) on $\R$ (or on an alternative domain that should be clear from the context). The indicator function of a set $X$ is denoted by ${\bf 1}_X$. The space of continuous functions on $\R$ is denoted by $C(\R)$.

\section{Extremal problems and the proofs of Theorems \ref{Dirichlet1} and \ref{Thm2_Non-vanishing}} \label{Sec2_proofs}

\subsection{The delta extremal problems} \label{Two-Delta_Sec}For $G \in \{{\rm U, Sp, O, SO}({\rm even}), {\rm SO}(\rm{odd})\}$, and $\Delta >0$, let us assume throughout this section the validity of the claim that the Hilbert space $\mathcal{H}_{G, \pi\Delta}$ and the Paley-Wiener space $\mathcal{H}_{\pi \Delta}$ are the same (as sets) and have equivalent norms. This is proved in Section \ref{Sec2}. For each $t \in \R$, let us define the following classes of functions $M:\R \to \R$:
\begin{equation*}
\mathcal{A}^{\star}_{2\pi\Delta}(t):= \Big\{ M \in L^1(\R) \cap C(\R)\ ; \ {\rm supp}(\widehat{M}) \subset [-\Delta,\Delta] \ ; \  M \geq 0 \ {\rm on} \ \R \ ; \ M(t) \geq 1\Big\}.
\end{equation*}
and
\begin{equation*}
\mathcal{A}_{2\pi\Delta}(t):= \Big\{ M \in L^1(\R) \cap C(\R)\ ; \ {\rm supp}(\widehat{M}) \subset [-\Delta,\Delta] \ ; \  M \geq 0 \ {\rm on} \ \R \ ; \ M(t) \geq 1 \ ; \ M(-t) \geq 1\Big\}.
\end{equation*}
Note that, by Fourier inversion, each function $M \in L^1(\R) \cap C(\R)$ with ${\rm supp}(\widehat{M}) \subset [-\Delta,\Delta]$ is the restriction to $\R$ of an entire function of exponential type at most $2 \pi \Delta$. We consider here two extremal problems.

\medskip

\noindent{\it One-delta extremal problem.} For $G \in \{{\rm U, Sp, O, SO}({\rm even}), {\rm SO}(\rm{odd})\}$, $\Delta >0$ and $t \in \R$, find:
\begin{equation}\label{20210504_12:23_1}
\inf_{M \in\mathcal{A}^{\star}_{2\pi\Delta}(t)}  \int_\mathbb{R} M(x) \, W_G(x) \, \d x.
\end{equation}

\medskip

\noindent{\it Two-delta extremal problem.} For $G \in \{{\rm U, Sp, O, SO}({\rm even}), {\rm SO}(\rm{odd})\}$, $\Delta >0$ and $t \in \R$, find:
\begin{equation}\label{20210504_12:23}
\inf_{M \in\mathcal{A}_{2\pi\Delta}(t)}  \int_\mathbb{R} M(x) \, W_G(x) \, \d x.
\end{equation}

\medskip

The bridge that connects these two extremal problems to our Hilbert spaces is a classical decomposition result due to Krein \cite[p. 154]{A}: a function $M:\R \to \R$ that verifies $M \in L^1(\R) \cap C(\R)$,  ${\rm supp}(\widehat{M}) \subset [-\Delta,\Delta]$ and  $M \geq 0 \ {\rm on} \ \R$ can be written as $M(x) = |F(x)|^2$ with $F \in \mathcal{H}_{\pi \Delta} = \mathcal{H}_{G, \pi \Delta}$, and conversely. 

\medskip 

The one-delta extremal problem \eqref{20210504_12:23_1} can then be reformulated in the reproducing kernel Hilbert space $\mathcal{H}_{G, \pi \Delta}$, in which one wants to minimize the norm $\|F\|^2_{\mathcal{H}_{G, \pi \Delta}}  = \|F\|^2_{L^2(\R, W_G)}$, subject to the condition $|F(t)| \geq 1$. The solution now follows by an application of the Cauchy-Schwarz inequality using the reproducing kernel (this idea dates back at least to the work of Holt and Vaaler \cite{HV})
\begin{equation}\label{20210512_12:33}
1 \leq |F(t)|^2 = |\langle F, K(t, \cdot)\rangle |^2 \leq \|F\|^2_{L^2(\R, W_G)}\,\|K(t, \cdot)\|^2_{L^2(\R, W_G)} = \|F\|^2_{L^2(\R, W_G)}\, K(t, t),
\end{equation}
where $K = K_{G, \pi\Delta}$ and $\langle \cdot\,, \cdot \rangle$ is the inner product in $\mathcal{H}_{G, \pi \Delta}$.
The conclusion is that 
\begin{equation}\label{20210512_14:33}
\inf_{M \in\mathcal{A}^{\star}_{2\pi\Delta}(t)}  \int_\mathbb{R} M(x) \, W_G(x) \, \d x = \frac{1}{K(t,t)}.
\end{equation}
The infimum is attained and the unique extremal function $\mathscr{M}^{\star} = \mathscr{M}^{\star}_{G, \Delta, t}  \in\mathcal{A}^{\star}_{2\pi\Delta}(t)$ (from the condition of equality in the Cauchy-Schwarz inequality \eqref{20210512_12:33}) is given by
\begin{equation}\label{20210512_14:34}
\mathscr{M^{\star}}(x) = \frac{|K(t, x)|^2}{K(t,t)^2}.
\end{equation}
Note that this function is even when $t=0$; see \S\ref{Sub5_Prelim} below.

\smallskip

The solution of the two-delta extremal problem \eqref{20210504_12:23} falls under the scope of the general framework of Carneiro, Chandee, Littmann and Milinovich \cite[Lemma 13]{CCLM} (see also the related works \cite{Kelly, Litt, Sono,V}). We state here this result for the convenience of the reader.

\begin{lemma}$(${\rm cf.} \cite[Lemma 13]{CCLM}$)$\label{HS_geometric_lemma}
Let $H$ be a Hilbert space $($over $\C$$)$ with norm $\| \cdot \|$ and inner product $\langle \cdot\,, \cdot \rangle$. Let $v_1, v_2 \in H$ be two nonzero vectors $($not necessarily distinct$)$ such that $\|v_1\| = \|v_2\|$ and define
$$\J = \big\{ x \in H; \ |\langle x, v_1\rangle |  \geq 1 \ {\rm and} \ |\langle x, v_2\rangle |  \geq 1\big\}.$$
Then
\begin{equation*}
\min_{x \in \J}\|x\| = \left(\frac{2}{\big(\|v_1\|^2 + |\langle v_1, v_2\rangle|\big)}\right)^{1/2}.
\end{equation*}
The extremal vectors $y \in \J$ are given by:
\begin{enumerate}
\item[(i)] If $\langle v_1, v_2\rangle = 0$, then $y =   \left(\dfrac{2}{\big(\|v_1\|^2 + |\langle v_1, v_2\rangle|\big)}\right)^{\!\!1/2} \!\dfrac{(c_1v_1 + c_2v_2)}{\left\|v_1 + v_2\right\|}$, where $c_1, c_2 \in \C$ with $|c_1| = |c_2| =1$. 
\smallskip
\item[(ii)] If $\langle v_1, v_2\rangle \neq 0$, and $\langle v_1, v_2\rangle =  \,e^{-i\alpha} \,|\langle v_1, v_2\rangle|$, then $y =   \left(\dfrac{2}{\big(\|v_1\|^2 + |\langle v_1, v_2\rangle|\big)}\right)^{1/2}\,  \dfrac{c\,(e^{i\alpha } v_1 + v_2)}{\left\| e^{i\alpha} v_1 + v_2\right\|}$, 
where $c \in \C$ with $|c| =1$. 
\end{enumerate}
\end{lemma}

As before, by Krein's decomposition, the two-delta extremal problem \eqref{20210504_12:23} can be reformulated in the reproducing kernel Hilbert space $\mathcal{H}_{G, \pi \Delta}$, in which one wants to minimize the norm $\|F\|^2_{\mathcal{H}_{G, \pi \Delta}}  = \|F\|^2_{L^2(\R, W_G)}$, subject to the conditions
 $$|F(t)| = |\langle F, K(t, \cdot)\rangle | \geq 1 \ \  {\rm and} \ \  |F(-t)| = |\langle F, K(-t, \cdot)\rangle | \geq 1,$$
In \S\ref{Sub5_Prelim} below we verify a few properties of $K$, including the following:
\begin{itemize}
\item[(a)] For each $w,z \in \C$ we have $K(w,z) = K(-w,-z)$. In particular, $K(t,t) = \|K(t, \cdot)\|_{L^2(\R, W_G)}^2 = \|K(-t, \cdot)\|_{L^2(\R, W_G)}^2 = K(-t,-t) > 0$. 
\smallskip
\item[(b)] The function $z \to K(t, z)$ is real entire for each $t\in \R$. Hence $K(t, -t) = \langle K(t, \cdot), K(-t, \cdot)\rangle = \langle K(-t, \cdot), K(t, \cdot)\rangle =  K(-t,t)$ is real-valued. 
\end{itemize}
This puts us in position to apply Lemma \ref{HS_geometric_lemma} and conclude that     
\begin{equation}\label{20210504_14:13}
\inf_{M \in\mathcal{A}_{2\pi\Delta}(t)}  \int_\mathbb{R} M(x) \, W_G(x) \, \d x = \frac{2}{K(t,t) + |K(t,-t)|}.
\end{equation}
Lemma \ref{HS_geometric_lemma} tells us that this infimum is attained and the function $\mathscr{M} = \mathscr{M}_{G, \Delta, t}  \in\mathcal{A}_{2\pi\Delta}(t)$ given by\footnote{Recall that ${\rm sgn}:\R \to \R$ is defined by ${\rm sgn}(t) =1$, if $t>0$; ${\rm sgn}(0) = 0$; and ${\rm sgn}(t) =-1$, if $t<0$. In the case $K(t,-t) = 0$, we are choosing constants $c_1$ and $c_2$ in Lemma \ref{HS_geometric_lemma} (i) with ${\rm Re}(c_1\overline{c_2}) =0$ to arrive at \eqref{20210504_13:44}.} 
\begin{equation}\label{20210504_13:44}
\mathscr{M}(x) =  \frac{K(t,x)^2 +K(-t,x)^2  + 2\, {\rm sgn}(K(t,-t)) \, K(t,x)\, K(-t,x)}{\big(K(t,t) + |K(t,-t)|\big)^2}
\end{equation}
is an extremizer in general, being unique if $K(t, -t) \neq 0$. Note that $\mathscr{M}$ is an even function with $\mathscr{M}(\pm t) = 1$.  

\medskip

\noindent{\sc Remark:} It is interesting to observe how much the solution of the two-delta extremal problem improves over the simple superposition (addition) of the two solutions of the one-delta extremal problem at $t$ and $-t$.  The latter construction would give an answer of $2/K(t,t)$ instead of the right-hand side of \eqref{20210504_14:13}. The improvement is more significant when $|t|$ is small, as illustrated in Figure \ref{figure2}.

\subsection{Proof of Theorem \ref{Thm2_Non-vanishing}} \label{Approx_Arg} Let us start with part (ii). Fix $t >0$. Using \eqref{assumption} and \eqref{duality}, for any non-negative function $\phi :\R \to \R$ with $\phi(\pm t)\geq 1$ we have 
\begin{align}\label{20210504_14:04}
\begin{split}
\frac{2}{|\mathcal{F}(Q)|} \sum_{f \in \mathcal{F}(Q)} \ \underset{s=\frac{1}{2}}{\textup{ord}} \ L\left(s+ \frac{2 \pi i t }{\log c_f}\,,\, f\right)  & =  \frac{1}{|\mathcal{F}(Q)|}  \sum_{f \in \mathcal{F}(Q)} \left\{ \underset{s=\frac{1}{2}}{\textup{ord}} \ L\left(s+ \frac{2 \pi i t }{\log c_f}\,,\, f\right)  + \underset{s=\frac{1}{2}}{\textup{ord}} \ L\left(s- \frac{2 \pi i t }{\log c_f}\,,\, \overline{f}\right) \right\} \\
& = \frac{1}{|\mathcal{F}(Q)|}  \sum_{f \in \mathcal{F}(Q)} \left\{\underset{s=\frac{1}{2}}{\textup{ord}} \ L\left(s+ \frac{2 \pi i t }{\log c_f}\,,\, f\right)  + \underset{s=\frac{1}{2}}{\textup{ord}} \ L\left(s- \frac{2 \pi i t }{\log c_f}\,,f\right) \right\}\\
& \leq \frac{1}{|\mathcal{F}(Q)|} \sum_{f \in \mathcal{F}(Q)} \sum_{\gamma_f} \phi\!\left( \gamma_f \frac{\log c_f}{2\pi}\right).
\end{split}
\end{align}
Assuming that estimate \eqref{20210430_11:41} holds for such $\phi$ we would have, from \eqref{20210504_14:04} and  \eqref{20210430_11:41}, 
\begin{align}\label{20210504_16:02}
\limsup_{ Q \to\infty} \,\frac{2}{|\mathcal{F}(Q)|} \sum_{f \in \mathcal{F}(Q)} \ \underset{s=\frac{1}{2}}{\textup{ord}} \ L\left(s+ \frac{2 \pi i t }{\log c_f}\,,\, f\right) \leq \limsup_{Q \to\infty}  \frac{1}{|\mathcal{F}(Q)|} \sum_{f \in \mathcal{F}(Q)} \sum_{\gamma_f} \phi\!\left( \gamma_f \frac{\log c_f}{2\pi}\right) = \int_\mathbb{R} \phi(x) \, W_G(x) \, \d x,
\end{align}
and this is how our two-delta extremal problem arises, in order to minimize the right-hand side above. Hence, if estimate \eqref{20210430_11:41} holds for the extremal function $\mathscr{M}$ defined in \eqref{20210504_14:13} - \eqref{20210504_13:44}, we are done. Although $\mathscr{M}$ is an entire function, it may not belong to the Schwartz class as we are assuming in the hypotheses of the theorem. This is no major concern as we can proceed by standard approximation arguments. For instance, let $\varphi: \R \to \R$ be an even, non-negative function, with ${\rm supp}(\varphi) \subset [-1,1]$;\, $\int_{\R}\varphi(y)\,\dy = 1$; and $\widehat{\varphi}$ non-negative. For $\varepsilon >0$, let $\varphi_{\varepsilon}(y): = \tfrac{1}{\varepsilon} \varphi \big( \tfrac{y}{\varepsilon}\big)$. For $0 < \delta < 1$ let $\mathscr{M}^{\delta}(x):= \delta \mathscr{M}(\delta x)$. Then $\widehat{\mathscr{M}^{\delta}}(y) = \widehat{\mathscr{M}}\big( \tfrac{y}{\delta}\big)$ and one sees that ${\rm supp} \big(\widehat{\mathscr{M}^{\delta}}\big) \subset [-\delta \Delta, \delta \Delta]$. Hence, if $0 < \varepsilon < \Delta - \delta \Delta$ we have that ${\rm supp} \big(\widehat{\mathscr{M}^{\delta}} * \varphi_{\varepsilon}\big) \subset (-\Delta, \Delta)$. For $0 < \varepsilon < \Delta - \delta \Delta$ let 
$$\phi_{\varepsilon}^{\delta}(x) = \frac{\mathscr{M}^{\delta}(x)\widehat{\varphi_{\varepsilon}}(x)}{\mathscr{M}^{\delta}(t)\widehat{\varphi_{\varepsilon}}(t)}.$$
Note that $\widehat{\varphi_{\varepsilon}}(t) = \widehat{\varphi}(\varepsilon t)$, and we can always choose $\varepsilon$ small enough so that this quantity is positive (recall that $\widehat{\varphi}(0) =1 $ and $t$ is fixed). Hence, $\phi_{\varepsilon}^{\delta}$ is an even and non-negative Schwartz function, with $\phi_{\varepsilon}^{\delta}(\pm t) =1$ and ${\rm supp} \big( \widehat{\phi_{\varepsilon}^{\delta}}\big) \subset (-\Delta, \Delta)$. In particular, we can plug $\phi_{\varepsilon}^{\delta}$ in the mechanism \eqref{20210504_14:04} - \eqref{20210504_16:02}. One can verify that, given any $\eta >0$, it is possible to choose $0 < \delta < 1$ ($\delta$ close to $1$), and  $0 < \varepsilon < \Delta - \delta \Delta$ ($\varepsilon$ close to $0$) such that 
$$ \int_\mathbb{R} \phi_{\varepsilon}^{\delta}(x) \, W_G(x) \, \d x \leq \left(\int_\mathbb{R} \mathscr{M}(x) \, W_G(x) \, \d x \right)+ \eta.$$
This concludes the proof of part (ii).

\smallskip

Part (i) is simpler. For any non-negative function $\phi :\R \to \R$ with $\phi(0)\geq 1$ we have
\begin{align}\label{20210513_14:26}
\frac{1}{|\mathcal{F}(Q)|} \sum_{f \in \mathcal{F}(Q)} \ \underset{s=\frac{1}{2}}{\textup{ord}} \ L\left(s\,,\, f\right) \leq \frac{1}{|\mathcal{F}(Q)|} \sum_{f \in \mathcal{F}(Q)} \sum_{\gamma_f} \phi\!\left( \gamma_f \frac{\log c_f}{2\pi}\right).
\end{align}
If estimate \eqref{20210430_11:41} holds for such $\phi$ we would have, from \eqref{20210513_14:26} and  \eqref{20210430_11:41},
\begin{align}\label{20210512_14:37}
\limsup_{Q \to\infty} \,\frac{1}{|\mathcal{F}(Q)|} \sum_{f \in \mathcal{F}(Q)} \ \underset{s=\frac{1}{2}}{\textup{ord}} \ L\left(s\,,\, f\right) \leq \limsup_{Q \to\infty}  \frac{1}{|\mathcal{F}(Q)|} \sum_{f \in \mathcal{F}(Q)} \sum_{\gamma_f} \phi\!\left( \gamma_f \frac{\log c_f}{2\pi}\right) = \int_\mathbb{R} \phi(x) \, W_G(x) \, \d x.
\end{align}
This brings us to the one-delta extremal problem at $t=0$. If estimate \eqref{20210430_11:41} holds for the extremal function $\mathscr{M}^{\star}$ defined in \eqref{20210512_14:33} - \eqref{20210512_14:34}, we are done. If not, we proceed with an approximation argument as above.

\smallskip

\noindent {\sc Remark:} Without the duality hypotheses \eqref{assumption} -- \eqref{duality} one can still proceed as in \eqref{20210513_14:26} - \eqref{20210512_14:37} (provided that estimate \eqref{20210430_11:41} holds for Schwartz functions $\phi:\R \to \R$ with ${\rm supp} (\widehat{\phi}) \subset (-\Delta, \Delta)$) using the one-delta extremal problem to prove, for any $t \in \R$, that
\begin{equation}\label{20210903_14:08}
\limsup_{Q \to\infty} \,\frac{1}{|\mathcal{F}(Q)|} \sum_{f \in \mathcal{F}(Q)} \ \underset{s=\frac{1}{2}}{\textup{ord}} \ L\left(s+ \frac{2 \pi i t }{\log c_f}\,,\, f\right) \leq \frac{1}{K(t,t)}.
\end{equation}
Note that \eqref{20210503_09:43} is generically sharper than \eqref{20210903_14:08}, for $t \neq 0$.

\subsection{Proof of Theorem \ref{Dirichlet1}} Theorem \ref{Dirichlet1} is a particular case of Theorem \ref{Thm2_Non-vanishing}, but we briefly give here the details. For $q$ prime, and an even and continuous function $\phi:\R \to \R$ with $\mathrm{supp}(\widehat{\phi}) \subseteq [-2,2]$ and such that $|\phi(x)| \ll (1+|x|)^{-1-\delta}$ for some $\delta>0$ as $|x|\to \infty$, Hughes and Rudnick \cite[Theorem 3.1]{HR} proved that
\begin{equation}\label{Hughes}
\begin{split}
\frac{1}{q\!-\!2}\sum_{\substack{\chi \, (\textup{mod }q) \\ \chi \ne \chi_0}} \sum_{\gamma_\chi} \phi\left(\gamma_\chi\frac{\log q}{2\pi}\right) &= \int_\mathbb{R} \phi(x) \, \dx + O\left(\frac{1}{\log q}\right),
\end{split}
\end{equation}
where $\frac{1}{2}+i\gamma_\chi$ runs over the non-trivial zeros of $L(s,\chi)$, counting multiplicity. This is estimate \eqref{20210430_11:41} in the regime $G = {\rm U}$ and $\Delta =2$. In this Paley-Wiener space, we have $K = K_{{\rm U}, 2\pi}$ given by
$$K(w,z) = \frac{\sin 2\pi (z - \overline{w})}{\pi (z   -  \overline{w})}.$$
Hence, our extremal function $\mathscr{M}$ defined in \eqref{20210504_14:13} - \eqref{20210504_13:44} is (recall that $t >0$ is fixed)
\begin{equation}\label{20210504_17:27}
\mathscr{M}(x) =  \frac{\left(\frac{\sin 2\pi (x - t)}{\pi (x   -  t)}\right)^2 + \left(\frac{\sin 2\pi (x + t)}{\pi (x   +  t)}\right)^2 + 2\, {\rm sgn}\big(\!\sin 4\pi t\big)\left(\frac{\sin 2\pi (x - t)}{\pi (x   -  t)}\right)\left(\frac{\sin 2\pi (x + t)}{\pi (x   +  t)}\right)}{\left(2 + \left|\frac{\sin 4 \pi t}{2\pi t}\right|\right)^2}.
\end{equation}
Note that $|\mathscr{M}(x)| \ll |x|^{-2}$ as $|x| \to \infty$. Proceeding as in \eqref{20210504_14:04} and using \eqref{Hughes} (with $\phi = \mathscr{M}$) one arrives at 
\begin{align*}
\frac{1}{q\!-\!2}\sum_{\substack{\chi \, (\textup{mod }q) \\ \chi \ne \chi_0}} \ \underset{s=\frac{1}{2}}{\textup{ord}} \ L\left(s+ \frac{2 \pi i t }{\log q},\chi\right)& \ \le \frac{1}{2(q\!-\!2)}\sum_{\substack{\chi \, (\textup{mod }q) \\ \chi \ne \chi_0}} \sum_{\gamma_\chi} \, \mathscr{M}\left(\gamma_\chi\frac{\log q}{2\pi}\right) = \frac{1}{2}\int_\mathbb{R} \mathscr{M}(x) \, \dx + O\left(\frac{1}{\log q}\right)\\
&   = \frac{1}{2} \, \left(1+\left|\frac{\sin 4 \pi t }{4\pi t}\right|\right)^{\!-1} + \ O\left(\frac{1}{\log q}\right).
\end{align*}
This completes the proof of Theorem \ref{Dirichlet1}.

\begin{figure} 
\includegraphics[scale=.6]{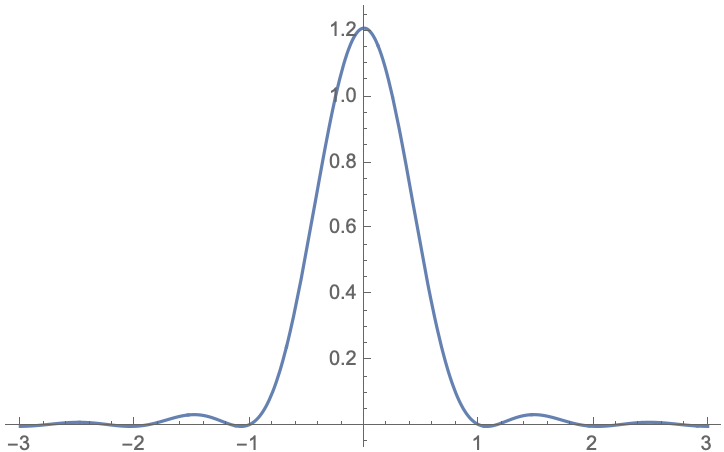} 
\includegraphics[scale=.6]{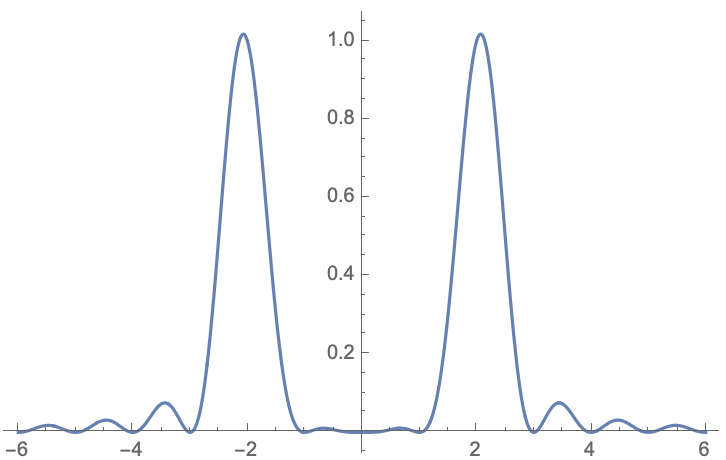} 
\caption{Plots of the extremal function $\mathscr{M}$ in \eqref{20210504_17:27} when $t = 1/4$ (on the left) and when $t = 2$ (on the right).}
\label{figure2}
\end{figure}

\section{Hilbert spaces and Fourier uncertainty}\label{Sec2}

We now establish the equivalence of norms between the spaces $\mathcal{H}_{G, \pi\Delta}$ and $\mathcal{H}_{\pi \Delta}$ defined in \S \ref{RKHS_Def}. We follow the outline of \cite[Lemma 12]{CCLM}, where a similar equivalence of norms was proved for Hilbert spaces associated to the pair correlation of zeros of the Riemann zeta-function. The fundamental tool for this purpose is the uncertainty principle for the Fourier transform, which appears in many different forms in the literature. We recall here the version of Donoho and Stark \cite{DS}, which is particularly well-suited for our argument.

\begin{lemma}{\rm (cf.~\cite[Theorem 2]{DS})}\label{uncertainty} Let $S, T \subset \R$ be measurable sets and let $F \in L^2(\R)$ with $\|F\|_{L^2(\R)} \!=\!1$. Then
$$|S|^{1/2}\,.\,|T|^{1/2} \geq  1 - \|F \, .\, {\bf 1}_{\R\setminus T}\|_{L^2(\R)} - \|\widehat{F} \, .\, {\bf 1}_{\R\setminus S}\|_{L^2(\R)} ,$$
where $|S|$ denotes the Lebesgue measure of the set $S$. 
\end{lemma}

The following qualitative result is enough for our purposes, and we present it here in the main text for its simplicity. With some additional work, it is possible to go further and obtain the sharp forms of the inequalities presented in \eqref{20210505_11:03}, below. We discuss such matters in an appendix at the end of the paper. 

\begin{proposition} \label{Lem_equiv_norms} Let $G \in \{{\rm U, Sp, O, SO}({\rm even}), {\rm SO}(\rm{odd})\}$ and $\Delta >0$.
Let $F:\C \to \C$ be an entire function. Then $F\in \mathcal{H}_{\pi \Delta}$ if and only if $F\in \mathcal{H}_{G, \pi\Delta}$.  Moreover, there exist positive constants $C^+$ and $C^- $, that may depend on $\Delta$, such that 
\begin{equation}\label{20210505_11:03}
C^- \,\|F\|_{L^2(\R)} \le  \|F\|_{L^2(\R, W_G)} \le  C^+ \, \|F\|_{L^2(\R)}
\end{equation}
for all $F\in \mathcal{H}_{\pi \Delta}$.
\end{proposition}
\begin{proof}
If $F\in \mathcal{H}_{\pi \Delta}$, then ${\rm supp}(\widehat{F}) \subset [-\tfrac{\Delta}{2}, \tfrac{\Delta}{2}]$. By Fourier inversion, the Cauchy-Schwarz inequality, and Plancherel's theorem, we get
\begin{equation}\label{20210505_10:58}
|F(0)|^2 = \left| \int_{-\frac{\Delta}{2}}^{\frac{\Delta}{2}} \widehat{F}(y)\, \dy \,\right|^2\leq \Delta \big\|  \widehat{F}\big\|^2_{L^2(\R)} = \Delta \|F\|^2_{L^2(\R)}.
\end{equation}
Hence, using the trivial bound $\big| \frac{\sin 2\pi x}{2\pi x} \big| \leq 1$, it follows from \eqref{densities} and \eqref{20210505_10:58} that 
\begin{align*}
\|F\|_{L^2(\R, W_G)}^2 \leq (2 + \Delta)\, \|F\|_{L^2(\R)}^2,
\end{align*}
for any of the symmetry groups $G$. This shows that $\!F \in \mathcal{H}_{G, \pi\Delta}$ and verifies the rightmost inequality in \eqref{20210505_11:03}.

\smallskip

Conversely, assume now that $F \in \mathcal{H}_{G, \pi\Delta}$. Since $F$ is entire, it is in particular continuous at the origin, and this leads us to $\|F\|_{L^2(\R)} < \infty$. Hence $F\in \mathcal{H}_{\pi \Delta}$. It remains to show the existence of a positive constant $C^-$ in \eqref{20210505_11:03}, independent of $F$. Letting $T = [-\tfrac{1}{8\Delta}, \tfrac{1}{8\Delta}]$ and $S =  [-\tfrac{\Delta}{2}, \tfrac{\Delta}{2}]$ in Lemma \ref{uncertainty} we plainly get
\begin{equation}\label{20210505_11:32}
\frac12 \|F\|_{L^2(\R)} \leq \|F \, .\, {\bf 1}_{\R\setminus T}\|_{L^2(\R)}.
\end{equation}
Let $0 < \eta < 1$ ($\eta$ may depend on $\Delta$) be such that 
\begin{equation}\label{20210505_11:33}
\eta^2\, {\bf 1}_{\R\setminus T}(x) \le  1 - \left|\frac{\sin 2\pi x}{2\pi x}\right|.
\end{equation}
Then, for any of the symmetry groups $G$, from \eqref{20210505_11:32}, \eqref{densities} and \eqref{20210505_11:33} we have
\begin{align*}
\frac{\eta}{2} \,\|F\|_{L^2(\R)} \leq \eta\, \|F \, .\, {\bf 1}_{\R\setminus T}\|_{L^2(\R)} \leq \|F\|_{L^2(\R, W_G)}.
\end{align*}
This verifies the leftmost inequality in \eqref{20210505_11:03} and concludes the proof.
\end{proof}

\section{Reproducing kernels}\label{Sec_Rep_Kernels}
In this section we establish the explicit representations for the reproducing kernels presented in Theorems \ref{Thm3_PW} -- \ref{Thm_odd_orthogonal}. Throughout the section we continue to let $G \in \{{\rm U, Sp, O, SO}({\rm even}), {\rm SO}(\rm{odd})\}$ and $\Delta >0$.

\subsection{Basic properties}\label{Sub5_Prelim} We start our discussion by proving a few basic properties of the reproducing kernel $K = K_{G, \pi\Delta}$ that were already used in \S \ref{Two-Delta_Sec}. These essentially follow from the definition \eqref{20210504_11:59} and the fact that $W_G(x)$ is even and real-valued (as a distribution). Let $\langle \cdot\,, \cdot \rangle$ and $\| \cdot\|$ be the inner product and the norm in the Hilbert space $\mathcal{H}_{G, \pi \Delta}$, respectively. 

\smallskip

\noindent (P1) We have observed that, for each $w \in \C$, we have $K(w,w) = \|K(w, \cdot)\|^2 \geq 0$. If, for some $w \in \C$, we had $K(w,w) = 0$, this would mean that $K(w, \cdot) = {\bf 0}$ and hence $F(w) = \langle F, K(w, \cdot)\rangle =0$ for all $F \in \mathcal{H}_{G, \pi \Delta}$. However, the function $F(z) = \frac{\sin (\pi\Delta (z - w))}{\pi \Delta(z-w)}$ belongs to $\mathcal{H}_{\pi \Delta}$ (and hence to $\mathcal{H}_{G, \pi \Delta}$) and verifies $F(w)= 1$, a contradiction. Therefore, we must have $K(w,w) >0$ for all $w \in \C$. 

\smallskip

\noindent (P2) If $F \in \mathcal{H}_{G, \pi \Delta}$ and we let $H(z) = F(-z)$, using the fact that the density function $W_G$ is even we get
$$F(w) = H(-w) = \int_{\R} H(x) \, \overline{K(-w, x)} \,W_G(x)\,\dx  =\int_{\R} F(x) \, \overline{K(-w, -x)} \,W_G(x)\,\dx.$$
From the definition \eqref{20210504_11:59} of the reproducing kernel (and its uniqueness), we must have, for each $w, z \in \C$,
$$K(-w, -z) = K(w, z).$$

\smallskip

\noindent (P3) For each $w, z \in \C$, with $F = K(z, \cdot)$ in definition \eqref{20210504_11:59} we get
\begin{equation}\label{20210510_14:27}
K(z,w) = \int_{\R} K(z,x) \, \overline{K(w, x)} \,W_G(x)\,\dx = \overline{\int_{\R} K(w,x) \, \overline{K(z, x)} \,W_G(x)\,\dx} = \overline{K(w,z)}. 
\end{equation}

\smallskip

\noindent (P4) If $F \in \mathcal{H}_{G, \pi \Delta}$ then $F^*(z) := \overline{F(\overline{z})} \in \mathcal{H}_{G, \pi \Delta}$. Hence
\begin{align*}
F(w) = \overline{F^*(\overline{w})}  & = \overline{\int_{\R} F^*(x) \, \overline{K(\overline{w}, x)} \,W_G(x)\,\dx} =  \int_{\R} F(x) \, K(\overline{w}, x) \,W_G(x)\,\dx.
\end{align*}
From the uniqueness of the reproducing kernel, we are led to the conclusion that 
$$\overline{K(\overline{w}, \overline{z})} = K(w,z)$$
for each $w, z \in \C$. In particular, for $t \in \R$ the function $z \mapsto K(t,z)$ is real entire.

\smallskip

We are now ready to move into challenge of describing the reproducing kernels explicitly.

\subsection{Cases $G \in \{ {\rm U, O}\}$ and $\Delta >0$} When $G = {\rm U}$ and $\Delta >0$, we have $\mathcal{H}_{{\rm U}, \pi \Delta}$ being the classical Paley-Wiener space $\mathcal{H}_{\pi \Delta}$ for which the reproducing kernel can be easily computed via Fourier inversion:
\begin{equation*}
K_{{\rm U},\pi\Delta}(w,z)= K_{\pi\Delta}(w,z) = \dfrac{\sin\pi\Delta(z-\overline{w})}{\pi(z-\overline{w})}.
\end{equation*}
When $G = {\rm O}$ and $\Delta >0$, we need an adjustment to incorporate the term $\tfrac12 \boldsymbol{\delta}_0(x)$ in the definition \eqref{densities}. In fact, a direct verification in \eqref{20210504_11:59} shows that 
\begin{align}\label{20210505_16:54}
\begin{split}
K_{{\rm O},\pi\Delta}(w,z)& = K_{\pi\Delta}(w,z) - \bigg(\dfrac{K_{\pi\Delta}(w,0)}{2+\Delta}\bigg)K_{\pi\Delta}(0,z)\\
& = \dfrac{\sin\pi\Delta(z-\overline{w})}{\pi(z-\overline{w})} - \frac{1}{(2 + \Delta)} \left(\dfrac{\sin\pi\Delta z}{\pi z}\right) \left( \dfrac{\sin\pi\Delta \overline{w} }{\pi \overline{w}}\right).
\end{split}
\end{align} 
This establishes Theorems \ref{Thm3_PW} and \ref{Rep_Kernel_O}.

\subsection{Cases $G \in \{{\rm Sp, SO(even), SO(odd)}\}$ and $0 < \Delta \leq 1$} If $G \in \{{\rm SO(even), SO(odd)}\}$, observe from \eqref{20210506_11:40_1} that  
\begin{equation}\label{20210506_11:50}
\widehat{W_{G}}(y) = \widehat{W_{{\rm O}}}(y) \ \ {\rm for} \ \ y \in [-1,1].
\end{equation}
Hence, if $0 < \Delta \leq 1$, directly from the definition \eqref{20210504_11:59}, identity \eqref{20210506_11:42} (applied to $\phi(x) = F(x)\,\overline{K_{{\rm G},\pi\Delta}(w,x)}$, that verifies ${\rm supp}\big(\widehat{\phi}\,\big) \subset [-\Delta, \Delta]$ when $F \in \mathcal{H}_{{\rm G}, \pi \Delta}$), and \eqref{20210506_11:50}, we conclude that
\begin{align*}
\begin{split}
K_{{\rm SO(even)},\pi\Delta}(w,z) & = K_{{\rm SO(odd)},\pi\Delta}(w,z)  = K_{{\rm O},\pi\Delta}(w,z) \\
& = \dfrac{\sin\pi\Delta(z-\overline{w})}{\pi(z-\overline{w})} - \frac{1}{(2 + \Delta)} \left(\dfrac{\sin\pi\Delta z}{\pi z}\right) \left( \dfrac{\sin\pi\Delta \overline{w} }{\pi \overline{w}}\right).
\end{split}
\end{align*}
When $G = {\rm Sp}$ we notice that 
\begin{equation}\label{20210506_12:02}
\widehat{W}_{{\rm Sp}}(y) \, = \, \boldsymbol{\delta}_0(y) - \tfrac{1}{2} \,{\bf 1}_{[-1,1]}(y) = \boldsymbol{\delta}_0(y) - \tfrac12 = \widehat{\big({\bf 1} - \tfrac12\,\boldsymbol{\delta}_0\big)}(y) \ \ {\rm for} \ \ y \in [-1,1].
\end{equation}
Hence, for $0 < \Delta \leq 1$, using \eqref{20210506_11:42} and \eqref{20210506_12:02}, we may perform an adjustment similar to \eqref{20210505_16:54} to conclude that  
\begin{align}\label{20210506_12:25}
\begin{split}
K_{{\rm Sp},\pi\Delta}(w,z)& = K_{\pi\Delta}(w,z) + \bigg(\dfrac{K_{\pi\Delta}(w,0)}{2-\Delta}\bigg)K_{\pi\Delta}(0,z)\\
& = \dfrac{\sin\pi\Delta(z-\overline{w})}{\pi(z-\overline{w})} + \frac{1}{(2 - \Delta)} \left(\dfrac{\sin\pi\Delta z}{\pi z}\right) \left( \dfrac{\sin\pi\Delta \overline{w} }{\pi \overline{w}}\right).
\end{split}
\end{align} 
This establishes Theorems \ref{Rep_Kernel_SO(even)}, \ref{Thm_symplectic} and \ref{Thm_odd_orthogonal} in the regime $0 < \Delta \leq 1$. 

\subsection{Relation between {\rm SO(odd)} and {\rm Sp}} From \eqref{densities} we have $W_{{\rm SO}({\rm odd})}(x) =   W_{\rm Sp}(x) +\boldsymbol{\delta}_0(x)$. In such a situation, where the two densities differ by a Dirac delta, one can relate the reproducing kernels, as already exemplified in \eqref{20210505_16:54} and \eqref{20210506_12:25}. In this particular case, one can check directly from the definition \eqref{20210504_11:59} that, for any $\Delta >0$, 
\begin{align*}
K_{{\rm SO(odd)},\pi\Delta}(w,z) = K_{{\rm Sp},\pi\Delta}(w,z) - \frac{K_{{\rm Sp},\pi\Delta}(w,0)}{1 + K_{{\rm Sp},\pi\Delta}(0,0)}K_{{\rm Sp},\pi\Delta}(0,z).
\end{align*}
This concludes the proof of Theorem \ref{Thm_odd_orthogonal}.

\subsection{Interlude: considerations from Fredholm theory} \label{SubS_Interlude} We are left with the harder task of finding explicit representations for the reproducing kernels in the cases  $G \in \{{\rm Sp, SO(even)}\}$ and $\Delta >1$. We start with a useful auxiliary lemma.

\begin{lemma}\label{Lem8}
Let $\Delta >0$ and $w \in \C$. 
\begin{itemize}
\item[(i)] There exist unique functions $u_w^+$ and $u_w^-$ in $L^2(\R)$, with ${\rm supp}\,(u_w^{\pm}) \subset \big[\!-\frac{\Delta}{2}, \frac{\Delta}{2}\big]$, such that 
\begin{align}\label{20210506_17:52}
	u^{\pm}_w(y)\pm \dfrac{1}{2}\int_{y-1}^{y+1}u^{\pm}_w(s)\,\ds = e^{-2\pi iwy} \ \ \ {\rm for \ a.e.} \ \,y \in \big[\!-\tfrac{\Delta}{2}, \tfrac{\Delta}{2}\big].
\end{align}
\item[(ii)] Define entire functions $k_w^+$ and $k_w^-$ by
\begin{align*} 
k^{\pm}_w(z):=\int_{-\tfrac{\Delta}{2}}^{\tfrac{\Delta}{2}}u^{\pm}_w(y)\,e^{2\pi iyz}\, \dy. 
\end{align*}
Then 
\begin{equation}\label{20210506_18:28}
K_{{\rm SO(even)},\pi\Delta}(w,z) = \overline{k_w^+(\overline{z})} \ \ \ \ {\rm and}\ \ \ \ K_{{\rm Sp},\pi\Delta}(w,z) = \overline{k_w^-(\overline{z})}.
\end{equation}
\end{itemize}
\end{lemma}
\begin{proof} Throughout this proof let $I := \big[\!-\!\frac{\Delta}{2}, \frac{\Delta}{2}\big]$. 

\smallskip

\noindent Part (ii). Assume for a moment that we have proved part (i). The Fourier transform of the function $k^{\pm}_{w}(x)\big(1\pm \frac{\sin2\pi x}{2\pi x}\big)$ is $u^{\pm}_{w}(y) \pm \frac{1}{2} \big(u^{\pm}_{w}*{\bf 1}_{[-1,1]}\big)(y)$ and the latter, by \eqref{20210506_17:52}, agrees with $e^{-2\pi iwy}$ on $I$ (as $L^2$-functions). Therefore, by the multiplication formula for the Fourier transform, if $F \in \mathcal{H}_{\pi \Delta}$ (hence ${\rm supp}(\widehat{F}) \subset I$), we have
\begin{align*}
\int_{\R} F(x) \left( k^{\pm}_{w}(x)\left(1\pm \dfrac{\sin2\pi x}{2\pi x}\right)\right) \,\dx =  \int_{I} \widecheck{F}(y)\,e^{-2\pi iwy}\,\dy  = F(w).
\end{align*}
This leads us to \eqref{20210506_18:28}.

\smallskip

\noindent Part (i). The operator $T: L^2(I) \to L^2(I)$ defined by 
\begin{equation*}
(Tu)(y) := \frac{1}{2} \big(u*{\bf 1}_{[-1,1]}\big)(y) = \frac{1}{2}\int_{I}{\bf 1}_{[-1,1]}(y-s)\, u(s)\,\ds \ \ \ ({\rm for} \  y \in I)
\end{equation*}
is a Hilbert-Schmidt operator (and hence compact) since the kernel $(y,s) \mapsto \frac12\,{\bf 1}_{[-1,1]}(y-s)$ belongs to $L^2 (I \times I)$. In \eqref{20210506_17:52} we seek to solve the functional equations
$$(u \pm Tu)(y) = e^{-2\pi iwy} \ \ {\rm in} \ \ L^2(I).$$
By Fredholm's alternative, each of these equations has a unique solution in $L^2(I)$ if and only if the homogeneous equations 
$$u \pm Tu = {\bf 0} \ \ {\rm in} \ \ L^2(I)$$
admit only the trivial solution. Let us verify that this is indeed the case. Assume that $u \in L^2(I)$ is a solution of $u \pm Tu = {\bf 0}$. This means that 
\begin{equation}\label{20210506_19:04}
u(y) = \mp \dfrac{1}{2}\int_{I \cap [y-1, y+1]} u(s)\,\ds
\end{equation}
for a.e. $y \in I$. Hence, we may assume without loss of generality that $u$ is absolutely continuous on $I$, since the right-hand side of \eqref{20210506_19:04} is. We now argue with a maximum principle in mind. Let $M = \max_{y \in I}|u(y)|$, and assume that $M>0$. Let $y_0 \in I$ by the minimal value of the set $\{y \in I\ ; \ |u(y)| = M\}$. Then, from \eqref{20210506_19:04}, 
\begin{equation}\label{20210507_08:51}
M = |u(y_0)| \leq \dfrac{1}{2}\left(\int_{I \cap [y_0-1, y_0]} |u(s)|\,\ds  +\int_{I \cap [y_0, y_0+1]} |u(s)|\,\ds\right) < M,
\end{equation}
a contradiction. Note that, in our setup, the first of the integrals in \eqref{20210507_08:51} must be strictly less than $M$, while the second is less than or equal to $M$. Hence we must have $u = {\bf 0}$. This concludes the proof.
\end{proof}

Lemma \ref{Lem8} establishes that the functions $u_w^+$ and $u_w^-$, which are essentially the Fourier transforms of our reproducing kernels, are the functions in $L^2(I)$ given by
$$u_w^+ = ({\rm Id} + T)^{-1} \left(e^{-2\pi iw (\cdot)}\right) \ \ {\rm and} \ \ u_w^- = ({\rm Id} - T)^{-1} \left(e^{-2\pi iw (\cdot)}\right).$$
Inverting these operators is a task that involves a certain computational cost. Our proposed method has two stages. In the first stage, starting from equation \eqref{20210506_17:52}, one will successively differentiate and manipulate the equations in order to achieve a differential equation that must be satisfied by $u$ in each of a few subintervals of $I$. These are what we call the {\it descending steps}, and the difficulty is that the degree of the final differential equations satisfied by $u$ grow with the parameter $\Delta$. In the second stage, one must retrace the steps in the hierarchy of differentiation to figure out the constant terms that appeared (which, in our case, are in fact functions of the variable $w$). This ultimately leads to a linear system of equations (which is well-posed since Lemma \ref{Lem8} shows that a solution exists and is unique). These are what we call the {\it ascending steps}. In the next subsections, we run this method to solve the functional equation in the case $1 < \Delta \leq 2$ and similar ideas could be applied to treat the cases $\Delta >2$. Similar computational challenges and strategies appear in a related step in the approach of Freeman and Miller \cite{F, FM}, which is essentially equivalent to the discussion below in the particular case $w=0$ (see also the work of Bernard \cite{Bernard}).

\subsection{Case $G = {\rm SO(even)}$ and $1 < \Delta \leq 2$} \label{Subs_Delta=2} We now complete the proof of Theorem \ref{Rep_Kernel_SO(even)}. With the notation of Lemma \ref{Lem8}, fix $w \in \C$ and let $u = u_w^+$. We may assume that $u$ is an absolutely continuous function on the interval $[-\tfrac{\Delta}{2},\tfrac{\Delta}{2}]$, and that $u$ is zero on $\R \setminus [-\tfrac{\Delta}{2},\tfrac{\Delta}{2}]$ ($u$ is not necessarily a continuous function on $\R$). The function $u$ verifies 
 \begin{equation}\label{20210507_16:39_R1}
 u(y) + \dfrac{1}{2}\int_{y-1}^{y+1}u(s)\,\ds = e^{-2\pi iwy} \ \ \ {\rm for\ all} \ \  -\tfrac{\Delta}{2} \leq y \leq \tfrac{\Delta}{2}.
 \end{equation}
\subsubsection{Descending steps} By the fundamental theorem of calculus, since ${\rm supp}(u) \subset [-\tfrac{\Delta}{2},\tfrac{\Delta}{2}]$, we get 
\begin{align}
u'(y) - \tfrac12 u(y-1) &= -2\pi i w \,e^{-2\pi iwy} \ \ \ {\rm for}  \ \ 1 - \tfrac{\Delta}{2} < y <  \tfrac{\Delta}{2}; \label{20210507_11:33_R1}\\
u'(y)  &= -2\pi i w \,e^{-2\pi iwy} \ \ \ {\rm for}  \ \  \tfrac{\Delta}{2}-1 < y <  1-\tfrac{\Delta}{2}; \label{20210507_11:33_R2}\\
u'(y) + \tfrac12 u(y+1) &= -2\pi i w \,e^{-2\pi iwy} \ \ \ {\rm for}  \ \ -\tfrac{\Delta}{2} < y <  \tfrac{\Delta}{2}-1.\label{20210507_11:34_R1}
\end{align}
Note that, if $\Delta =2$, equation \eqref{20210507_11:33_R2} should be disregarded. The general solution of \eqref{20210507_11:33_R2} is 
$$u(y) = e^{-2 \pi i wy} + D(w) \ \ \ {\rm for}  \ \  \tfrac{\Delta}{2}-1 \leq y \leq  1-\tfrac{\Delta}{2}. $$
For $1 - \tfrac{\Delta}{2} < y <  \tfrac{\Delta}{2}$, by differentiating \eqref{20210507_11:33_R1}, and adding up half of \eqref{20210507_11:34_R1} with $y$ replaced by $y-1$, we get
\begin{equation}\label{20210507_11:48_R1}
u''(y) + \frac14 u(y) =  \big(-4\pi^2w^2 - \pi i w\,e^{2\pi iw}\big) \,e^{-2\pi iwy}.
\end{equation}
Assume for a moment that $w \neq \pm 1/4 \pi$. The general solution of \eqref{20210507_11:48_R1} is 
\begin{equation}\label{20210507_12:27_R1}
u(y) = A(w)\,e^{iy/2}+B(w)\,e^{-iy/2}+C(w)\,e^{-2\pi iwy} \ \ {\rm for} \ \ 1 - \tfrac{\Delta}{2} \leq y \leq  \tfrac{\Delta}{2},
\end{equation}
where
$$C(w)= \frac{-16\pi^2w^2 -4 \pi i w\,e^{2\pi iw}}{1 - 16 \pi^2w^2}.$$
An analogous reasoning shows that
\begin{equation}\label{20210507_12:28_R1}
u(y) = A_1(w)\,e^{iy/2}+B_1(w)\,e^{-iy/2}+ \overline{C(\overline{w})}\,e^{-2\pi iwy}  \ \ {\rm for} \ \ -\tfrac{\Delta}{2} \leq y \leq  \tfrac{\Delta}{2}-1.
\end{equation}

\subsubsection{Ascending steps} At this point, we must retrace the steps in the hierarchy of differentiation to figure out the constant terms (in the variable $y$). Using \eqref{20210507_12:27_R1} and \eqref{20210507_12:28_R1} in \eqref{20210507_11:33_R1} (or in \eqref{20210507_11:34_R1}) we get
\begin{equation*}
A_1(w) =i\,e^{i/2}\, A(w)  \ \ \ {\rm and} \ \ \ B_1(w) = -i\,e^{-i/2} \,B(w).
\end{equation*}
It is then enough to determine $A(w)$, $B(w)$ and $D(w)$. We need three pieces of information to set up a linear system of equations. There are different ways of doing this, by evaluating \eqref{20210507_16:39_R1} at particular points and/or by evaluating the continuity conditions of $u$ at the junction points of the internal subintervals. We choose the following system:
\begin{align}\label{20210510_11:18_R1}
\begin{cases}
a_1 A(w) + \overline{a_1} \,B(w) + \left( \tfrac{2 - \Delta}{2}\right) D(w)= E(w) \,;\\
b_1 A(w) + \overline{b_1} \,B(w) + \left( \tfrac{2 - \Delta}{2}\right) D(w)= \overline{E(\overline{w})}\,;\\
\tau A(w) + \overline{\tau} B(w) - 2D(w) = F(w),
\end{cases}
\end{align}
where $a_1  := e^{\Delta i/4} + i e^{(2-\Delta)i/4} - ie^{\Delta i/4}$\,;\, $b_1 :=e^{ \Delta i/4} +  ie^{ (2-\Delta)i /4} - e^{(2- \Delta) i/4}$\,;\, $\tau := e^{(2-\Delta)i/4}  + i e^{\Delta i/4}$, 
\begin{align*}
E(w) := \frac{2\cos(\pi\Delta w)+ 4 \pi i w e^{- \pi \Delta i w } - e^{\pi (2 -\Delta) i w }  }{1 - 16 \pi^2 w^2}  - \frac{\sin(\pi (2 - \Delta) w )}{2\pi w (1 - 16 \pi^2 w^2)}.
\end{align*}
and
\begin{align*}
F(w) := \frac{2\cos(\pi(2-\Delta)w) - 8 \pi w \sin(\pi \Delta w)}{1 - 16 \pi^2 w^2}.
\end{align*}
The first two equations in \eqref{20210510_11:18_R1} come from the evaluation of \eqref{20210507_16:39_R1} at the points $y=\tfrac{\Delta}{2}$ and $y=-\tfrac{\Delta}{2}$. The third equation in \eqref{20210510_11:18_R1} (which is not necessary if $\Delta =2$) comes from taking the continuity condition of $u$ at the points  $y = 1 - \tfrac{\Delta}{2}$ and $y = \tfrac{\Delta}{2} -1$ and adding them up to get matters in a slightly more symmetric form.

\smallskip

In \eqref{20210510_11:18_R1}, multiplying the third equation by $ \left( \tfrac{2 - \Delta}{4}\right)$ and adding it up to the first two equations yields
\begin{align*}
\begin{cases}
a A(w) + \overline{a} \,B(w) = G(w) \,;\\
b A(w) + \overline{b} \,B(w) = \overline{G(\overline{w})}\,,
\end{cases}
\end{align*}
with $a := a_1 + \left( \tfrac{2 - \Delta}{4}\right)\tau$\,;\, $b := b_1 + \left( \tfrac{2 - \Delta}{4}\right)\tau$\,; and 
\begin{equation*}
G(w):= E(w) + \left( \tfrac{2 - \Delta}{4}\right)F(w).
\end{equation*}
At this point we get 
\begin{equation}\label{20210514_10:59}
A(w) = \frac{\overline{a}\,\overline{G(\overline{w})} - \overline{b}\,G(w)}{\overline{a} b - a\overline{b}} \ \ ;  \ \ B(w) = \frac{b\,G(w) - a\,\overline{G(\overline{w})}}{\overline{a}b - a \overline{b}}.
\end{equation}
One can check that for $1 < \Delta \leq 2$ the denominator in \eqref{20210514_10:59} is indeed non-zero. We can then find $D(w)$ from the third equation of \eqref{20210510_11:18_R1}. 

\smallskip

We have now completely determined $u = u_w^+$. Then 
\begin{align*}
&k^+_w(z)  =\int_{-\frac{\Delta}{2}}^{\frac{\Delta}{2}}u(y)\,e^{2\pi iyz}\, \dy  = \int_{-\frac{\Delta}{2}}^{\frac{\Delta}{2}-1 }u(y)\,e^{2\pi iyz}\, \dy + \int_{\frac{\Delta}{2}-1}^{1-\frac{\Delta}{2} }u(y)\,e^{2\pi iyz}\, \dy + \int_{1-\frac{\Delta}{2}}^{\frac{\Delta}{2} }u(y)\,e^{2\pi iyz}\, \dy\\
& =  \frac{A(w) \big(i e^{-2\pi i z} + 1\big)\big(e^{\Delta( \pi i z + i/4)} - e^{(2-\Delta)(\pi i z + i/4)}\big)}{2\pi i z + i/2} + \frac{B(w)\big(\!\!-i e^{-2\pi i z} + 1\big)\big(e^{\Delta( \pi i z - i/4)} - e^{(2-\Delta)(\pi i z - i/4)}\big)}{2\pi i z - i/2} \\
& \qquad \qquad + \frac{\overline{C(\overline{w})}\big(- e^{-\pi \Delta i (z - w)} + e^{-\pi(2-\Delta) i (z - w)}\big)   + C(w)\big(e^{\pi\Delta i (z-w)} - e^{\pi (2 - \Delta) i (z-w)}\big)}{2\pi i (z-w)}\\
&  \qquad \qquad +  \frac{D(w) \sin(\pi (2-\Delta) z)}{\pi z} + \frac{\sin(\pi (2-\Delta) (z - w))}{\pi(z-w)}.
\end{align*}
From $\eqref{20210506_18:28}$ we have $K_{{\rm SO(even)},\pi\Delta}(w,z) = \overline{k_w^+(\overline{z})}$. This leads us to the proposed explicit expression for $K_{{\rm SO(even)},\pi\Delta}(w,z)$ in Theorem \ref{Rep_Kernel_SO(even)}. Finally, from \eqref{20210510_14:27}, the function $w \mapsto K_{{\rm SO(even)},\pi\Delta}(\overline{w},z)$ is entire. Hence, the cases $w = \pm1/4\pi$ that we had left behind are removable singularities.

\subsection{Case $G = {\rm Sp}$ and $1 < \Delta \leq 2$} \label{subsec_Sp_1-2}The procedure is entirely analogous to what we did in \S\ref{Subs_Delta=2}, following it line by line, now with $u = u_w^-$ solution of 
 \begin{equation*}
 u(y) - \dfrac{1}{2}\int_{y-1}^{y+1}u(s)\,\ds = e^{-2\pi iwy} \ \ \ {\rm for\ all} \ \  -\tfrac{\Delta}{2} \leq y \leq \tfrac{\Delta}{2}.
 \end{equation*}
We omit most of the details for simplicity. One just has to be careful with some sign changes that will appear in the auxiliary constants and functions. In the statement of Theorem \ref{Thm_symplectic} we denote the auxiliary constants and functions with the same letters as in Theorem \ref{Rep_Kernel_SO(even)} to reinforce the indication that they play the exact same roles as in \S\ref{Subs_Delta=2}. In addition to the auxiliary variables stated in Theorem \ref{Thm_symplectic}, we have the following ones that appear along \S\ref{Subs_Delta=2}:
\begin{align*}
A_1(w) =-i\,e^{i/2}\, A(w)  \ \ \ &; \ \ \ B_1(w) = i\,e^{-i/2} \,B(w)\,;\\
a_1= e^{\Delta i /4} - i e^{(2-\Delta)i/4} + ie^{ \Delta i/4} \ \ \ &; \ \ \ b_1 = e^{\Delta i/4} - ie^{(2-\Delta) i/4} - e^{(2- \Delta) i/4}.
\end{align*}
The system of equations \eqref{20210510_11:18_R1} becomes
\begin{align*}
\begin{cases}
a_1 A(w) + \overline{a_1} \,B(w) - \left( \tfrac{2 - \Delta}{2}\right) D(w)= E(w) \,;\\
b_1 A(w) + \overline{b_1} \,B(w) - \left( \tfrac{2 - \Delta}{2}\right) D(w)= \overline{E(\overline{w})}\,;\\
\tau A(w) + \overline{\tau} B(w) - 2D(w) = F(w),
\end{cases}
\end{align*}
with 
\begin{align*}
E(w) := \frac{2\cos(\pi\Delta w)- 4 \pi i w e^{- \pi \Delta i w } - e^{\pi (2 -\Delta) i w }  }{1 - 16 \pi^2 w^2}  + \frac{\sin(\pi (2 - \Delta) w )}{2\pi w (1 - 16 \pi^2 w^2)}
\end{align*}
and
\begin{align*}
F(w) := \frac{2\cos(\pi(2-\Delta)w) + 8 \pi w \sin(\pi \Delta w)}{1 - 16 \pi^2 w^2}.
\end{align*}
This leads to the result stated in Theorem \ref{Thm_symplectic}.

 \section{De Branges spaces and the existence of low-lying zeros} 
This section brings a general discussion on the problem of establishing the {\it existence} of low-lying zeros for families of $L$-functions, generalizing the work of Hughes and Rudnick in \cite[Theorem 8.1]{HR} for all the symmetry types, providing an alternative approach to the work of Bernard \cite{Bernard}. We shall obtain Theorem \ref{Thm_Db_9} as a corollary of a much more general extremal result (see Theorem \ref{DB_Thm} below).

\subsection{The Hughes--Rudnick extremal problem} Recall that our goal is to bound the quantity 
\begin{equation*}
\beta(\mc{F}) := \limsup_{Q \to \infty} \ \min_{\substack{ \gamma_f \notin \mc{Z}(f) \\ f \in \mc{F}(Q)}} \left|\frac{\gamma_{f} \,\log c_f}{2\pi}\right|,
\end{equation*}
introduced in \S \ref{Intro_low_lying_zeros}. Suppose we have an even (real-valued) Schwartz function $\phi$ with ${\rm supp} (\widehat{\phi}) \subset (-\Delta, \Delta)$ that verifies $\phi(x) \leq 0$ for $|x| \leq a$, and $\phi(x) \geq 0$ for $|x| \geq a$, for some $a >0$. If
\begin{equation}\label{20210907_12:14am}
\lim_{Q \to\infty} \, \frac{1}{|\mathcal{F}(Q)|} \sum_{f\in \mathcal{F}(Q)} \, \sum_{\gamma_f \notin \mc{Z}(f)} \phi\!\left( \gamma_f \frac{\log c_f}{2\pi}\right) < 0,
\end{equation}
then we can conclude that $\beta(\mc{F}) \leq a$. Hughes and Rudnick \cite{HR} make the natural choice $\phi(x) = (x^2 - a^2)\,g(x)$, with $g$ even, Schwartz, non-negative and ${\rm supp} (\widehat{g}) \subset (-\Delta, \Delta)$ (and hence ${\rm supp} (\widehat{\phi}) \subset (-\Delta, \Delta)$ by the Paley-Wiener theorem). From \eqref{20210907_11:58}  and \eqref{20210907_12:14am} we then need
\begin{equation*}
0 > \lim_{Q\to\infty} \, \frac{1}{|\mathcal{F}(Q)|} \sum_{f\in \mathcal{F}(Q)} \sum_{\gamma_f \notin \mc{Z}(f)} \phi\!\left( \gamma_f \frac{\log c_f}{2\pi}\right) = \int_\mathbb{R} \phi(x) \, W_{G^\sharp}(x) \, \d x = \int_\mathbb{R} (x^2 - a^2)\,g(x) \, W_{G^\sharp}(x) \, \d x,
\end{equation*}
which is equivalent to 
\begin{equation}\label{20210907_14:30}
a > \dfrac{\left(\int_\mathbb{R} x^2 \,g(x) \, W_{G^\sharp}(x)  \, \d x\right)^{1/2}}{\left(\int_\mathbb{R} g(x) \, W_{G^\sharp}(x) \, \d x\right)^{1/2}}.
\end{equation}
Recalling Krein's decomposition theorem \cite[p.~154]{A}, that such $g$ must be of the form $g(x) = |F(x)|^2$ for some $F \in \mathcal{H}_{G^{\sharp}, \pi \Delta}$, this leads us to consider the following extremal problem: find
\begin{equation}\label{20210907_13:49}
{\bf A}_{G, \pi\Delta} := \inf_{{\bf 0} \neq F \in \mathcal{H}_{G^{\sharp}, \pi \Delta}}  \frac{\|z \, F\|_{L^2(\R, \, W_{G^\sharp})}}{\|F\|_{L^2(\R, \,W_{G^\sharp})}}.
\end{equation}

We shall see below that the infimum in \eqref{20210907_13:49} is attained by an even extremal function ${\bf 0} \neq \frak{F} \in\mathcal{H}_{G^{\sharp}, \pi \Delta}$. With an approximation argument along the same lines of \S \ref{Approx_Arg}, one may choose a Schwartz function $g$ (close to $|\frak{F}|^2$) in \eqref{20210907_14:30}, yielding any arbitrary $a > {\bf A}_{G, \pi\Delta}$. The conclusion is that 
\begin{equation}\label{20210914_13:55}
\beta(\mc{F}) \leq {\bf A}_{G, \pi\Delta}.
\end{equation}

\subsection{De Branges spaces} \label{DeBranges_Intro} Our aim is to connect the extremal problem \eqref{20210907_13:49} to the beautiful theory of de Branges spaces of entire functions \cite{Branges}. In order to do so, we start by very briefly reviewing some of the main elements of this theory, and we invite the reader to consult \cite[Chapters 1 and 2]{Branges} for the relevant additional details. Recall that, for an entire function $F :\C \to \C$, we define the entire function $F^*:\C \to \C$ by $F^*(z) := \overline{F(\overline{z})}$. We denote by $\C^+ = \{z \in \C \ ; \ {\rm Im}(z) >0\}$ the open upper half-plane.

\smallskip

Given a {\it Hermite-Biehler} function $E: \C \to \C$, i.e. an entire function that verifies $|E(\overline{z})| < |E(z)|$ for all $z \in \C^+$, the de Branges space $\mc{H}(E)$ associated to $E$ is the space of entire functions $F:\C \to \C$ such that
\begin{equation}\label{20210913_14:31}
\|F\|_{\mc{H}(E)}^2 := \int_{\R} |F(x)|^{2} \, |E(x)|^{-2} \, \dx <\infty\,,
\end{equation}
and such that $F/E$ and $F^*/E$ have bounded type and non-positive mean type\footnote{A function $F$, analytic in $\C^{+}$, has {\it bounded type} if it can be written as the quotient of two functions that are analytic and bounded in $\C^{+}$. If $F$ has bounded type in $\C^{+}$, from its Nevanlinna factorization \cite[Theorems 9 and 10]{Branges} one has $v(F) := \limsup_{y \to \infty} \, y^{-1}\log|F(iy)| <\infty$. The number $v(F)$ is called the {\it mean type} of $F$.} in $\C^{+}$. This is a reproducing kernel Hilbert space with inner product 
\begin{equation*}
\langle F, G \rangle_{\mc{H}(E)} :=  \int_{-\infty}^\infty F(x) \, \ov{G(x)} \, |E(x)|^{-2} \, \dx.
\end{equation*}
The reproducing kernel (that we keep denoting by $K(w,\cdot)$) is given by (see \cite[Theorem 19]{Branges})
\begin{equation}\label{20210809_11:30am}
2\pi i (\ov{w}-z)K(w,z) = E(z)E^*(\ov{w}) - E^*(z)E(\ov{w}). 
\end{equation}
Associated to $E$, one can consider a pair of real entire functions $A$ and $B$ such that $E(z) = A(z) -iB(z)$. These companion functions are given by
\begin{equation}\label{20210909_18:29}
A(z) := \frac12 \big(E(z) + E^*(z)\big) \ \ \ {\rm and}\ \ \  B(z) := \frac{i}{2}\big(E(z) - E^*(z)\big)\,, 
\end{equation}
and note that they can only have real roots, by the Hermite-Biehler condition. The reproducing kernel has the alternative representation 
\begin{equation}\label{20210913_13:57}
\pi (z - \ov{w})K(w,z) = B(z)A(\ov{w}) - A(z)B(\ov{w})\,,
\end{equation}
and, when $z = \ov{w}$, one has
\begin{equation}\label{Intro_Def_K}
\pi K(\ov{z}, z) = B'(z)A(z) - A'(z)B(z).
\end{equation}
We shall only be interested in the situation where $K(w,w) = \|K(w, \cdot)\|^2_{\mc{H}(E)} > 0$ for all $w \in \C$, which is equivalent (see \cite[Lemma 11]{HV}) to the statement that $E$ has no real zeros. In this case, from \eqref{Intro_Def_K}, one sees that all the roots of $A$ and $B$ are simple. The set of functions $\Gamma_A:= \{K(\xi, \cdot)\ ;  A(\xi) = 0\}$ is always an orthogonal set in $\mc{H}(E)$. If $A \notin \mc{H}(E)$, the set $\Gamma_A$ is an orthogonal basis of $\mc{H}(E)$ and, if $A \in \mc{H}(E)$, the only elements of $\mc{H}(E)$ that are orthogonal to $\Gamma_A$ are the constant multiples of $A$ (see \cite[Theorem 22]{Branges} with $\alpha = \pi/2$). In particular, if $A \notin \mc{H}(E)$, for every $F \in \mc{H}(E)$ we have
\begin{equation}\label{20210809_11:01}
F(z) = \sum_{A(\xi) = 0}  \frac{F(\xi)}{K(\xi, \xi)} \, K(\xi, z) \ \ \ {\rm and} \ \ \   \|F\|_{\mc{H}(E)}^2 = \sum_{A(\xi) = 0}  \frac{\big| F(\xi)\big|^2 }{K(\xi, \xi)}.
\end{equation}
The most basic example of a de Branges space is the classical Paley-Wiener space $\mc{H}_{\pi \Delta}$, in which one can take $E(z) = e^{-\pi \Delta i z}$, $A(z) = \cos(\pi \Delta z)$ and $B(z) = \sin(\pi \Delta z)$, and \eqref{20210809_11:01} is well-known by Fourier analysis methods. In full generality, \eqref{20210809_11:01} provides a remarkable extension of Plancherel's identity and plays an important role in our approach.

\smallskip

We now draw the reader's attention to another relevant ingredient in our strategy: the ability to construct a de Branges space that is isometric to a given reproducing kernel Hilbert space of entire functions. That is, instead of constructing $K$ from $E$ as in \eqref{20210809_11:30am}, sometimes it is also possible to construct $E$ from $K$ (not necessarily in a unique way). 

\smallskip

In our particular situation, letting $K = K_{G^{\sharp}, \pi\Delta}$ be the reproducing kernel of the Hilbert space $\mathcal{H}_{G^{\sharp}, \pi \Delta}$ we define the function $L = L_{G^{\sharp}, \pi\Delta}$ by
\begin{equation*}
L(w,z) := 2\pi i (\overline{w} - z) K(w,z)\,,
\end{equation*}
and the entire function $E = E_{G^{\sharp}, \pi\Delta}$ by 
\begin{equation}\label{20210909_19:00}
E(z) := \frac{L(i,z)}{L(i,i)^{\frac12}}.
\end{equation}
One can show that $E$ is a Hermite-Biehler function such that 
\begin{equation}\label{20210909_19:09}
L(w,z) = E(z)E^*(\ov{w}) - E^*(z)E(\ov{w})\,;
\end{equation}
see \cite[Appendix A]{CCLM} for the details. This implies \cite[Theorem 23]{Branges} that the Hilbert space $\mathcal{H}_{G^{\sharp}, \pi \Delta}$ is equal isometrically to the de Branges space $\mc{H}(E)$, which yields the key identity 
\begin{equation}\label{20210809_11:39am}
\|F\|^2_{L^2(\R, W_{G^{\sharp}})} = \|F\|^2_{\mathcal{H}_{G^{\sharp}, \pi \Delta}}  = \int_{\R} |F(x)|^2 \,W_{G^{\sharp}}(x)\,\dx = \int_{\R} |F(x)|^{2} \, |E(x)|^{-2}\,\dx = \|F\|_{\mc{H}(E)}^2.
\end{equation}
Note that one does not necessarily have $W_{G^{\sharp}}(x) = |E(x)|^{-2}$ a.e.. Writing $E(z) = A(z) -iB(z)$, with $A$ and $B$ as in \eqref{20210909_18:29}, from the basic properties of $K$ in \S\ref{Sub5_Prelim}, one plainly sees that $A$ is even (which is going to be important for our argument) and $B$ is odd. We also have that $A(0) \neq 0$, since otherwise we would have a double zero at $x=0$ (recall that $A$ is even), and by \eqref{Intro_Def_K} this would contradict the fact that $K(0,0)>0$.

\smallskip

\noindent {\sc Remark}: The construction of $E$ such that $\mathcal{H}_{G^{\sharp}, \pi \Delta}$ is equal isometrically to $\mc{H}(E)$ is not unique. In fact, for any $\alpha \in \C^+$, one can choose $E_\alpha(z):= L(\alpha,z)/L(\alpha, \alpha)^{\frac12}$ in place of \eqref{20210909_19:00}. This is a Hermite-Biehler function, identity \eqref{20210909_19:09} continues to hold (with $E$ replaced by $E_{\alpha}$) and the corresponding de Branges space $\mc{H}(E_{\alpha})$ is equal isometrically to $\mathcal{H}_{G^{\sharp}, \pi \Delta}$; see \cite[Appendix A]{CCLM}. We choose $\alpha = i$ for simplicity. Generally, the functions $E_{\alpha}$ are different for different values of $\alpha$, and one actually has a family of identities given by \eqref{20210809_11:39am}. If one writes $E_{\alpha} = A_{\alpha} - i B_{\alpha}$ as usual, one can show that whenever $\alpha = i t$, with $t >0$, then $A_{it}/A_i$ and $B_{it}/ B_i$ are real positive numbers (not necessarily the same). In particular $A_{it}$ and $A_i$ have the same zeros. When $\alpha$ is not purely imaginary, the companion function $A_\alpha$ is not necessarily even. For example, for the Paley-Wiener space $\mc{H}_{\pi\Delta}$, one starts with $K(w,z) = \sin (\pi\Delta (z - \ov{w}))/(\pi (z - \ov{w}))$ and hence $L(w,z) = -2i\sin (\pi\Delta(z - \ov{w}))$. If $\alpha = it$, with $t>0$, we have
\begin{equation}\label{20210909_22:49}
E_{it}(z)  =  \frac{2\sinh \pi \Delta t \, \cos \pi \Delta z - 2i \,\cosh \pi \Delta t\,\sin \pi \Delta z}{\big(2 \sinh 2\pi \Delta t\big)^{\frac{1}{2}}}.
\end{equation}
In this case, it is interesting to notice that none of these functions $E_{\alpha}$, for $\alpha \in \C^+$, is actually equal to the `classical' generator $E(z) = e^{-\pi \Delta i z}$. If we let $t \to \infty$ in \eqref{20210909_22:49} we would then recover the classical one.

\subsection{Solving the extremal problem in a broader setting} We are now in position to state and prove the main result of this section.
\begin{theorem} \label{DB_Thm}Let $E$ be a Hermite-Biehler function with no real zeros and such that $z \mapsto E(iz)$ is real entire.
Let $\mc{H}(E)$ be the associated de Branges space with reproducing kernel $K$. Let $A := \frac{1}{2} \big(E + E^*\big)$ and let $\xi_0$ be the smallest positive real zero of $A$. Then 
\begin{equation*}
{\bf A}_{E} := \inf_{{\bf 0} \neq F \in \mathcal{H}(E)} \frac{\|z \, F\|_{\mc{H}(E)}}{\|F\|_{\mc{H}(E)}} = \xi_0.
\end{equation*}
The unique extremizers are $\frak{F}(z) = c\big( K(\xi_0, z) + K(\xi_0, -z)\big)$, with $c \in \C\setminus\{0\}$. 
\end{theorem}

\begin{proof} Letting $B = \frac{i}{2} \big(E - E^*\big)$, the fact that $z \mapsto E(iz)$ is real entire is equivalent to the statement that $A$ is even and $B$ is odd. Note that $A(0) \neq 0$, otherwise we would have $K(0,0) =0$ by \eqref{Intro_Def_K}, which in turn would imply that $E(0) = 0$, a contradiction. In what follows we denote $\mc{X}(E) := \{F \in \mc{H}(E)\, : \,  zF \in \mc{H}(E)\}$. One can check directly from \eqref{20210913_13:57} that the proposed extremizers are non-zero functions in the subspace $\mc{X}(E)$.

\smallskip

We start with the most typical case $A \notin \mc{H}(E)$. Let ${\bf 0} \neq F \in \mc{X}(E)$. From \eqref{20210809_11:01} we get
\begin{align}\label{20210909_23:48}
\|F\|_{\mc{H}(E)}^2  = \sum_{A(\xi) = 0}  \frac{\big| F(\xi)\big|^2 }{K(\xi, \xi)} \leq \frac{1}{\xi_0^2} \sum_{A(\xi) = 0}  \frac{|\xi|^2\,\big|F(\xi)\big|^2 }{K(\xi, \xi)} = \frac{1}{\xi_0^2}\, \|z F\|_{\mc{H}(E)}^2. 
\end{align}
This plainly shows that ${\bf A}_{E} \geq \xi_0$. In order to have equality in \eqref{20210909_23:48} one must have $F(\xi) = 0$ for each $\xi \in \R$ such that $A(\xi) = 0$ and $\xi \neq \pm \xi_0$. By the interpolation formula in \eqref{20210809_11:01} (recall that $K(\xi_0, \xi_0) = K(-\xi_0, -\xi_0)$ in our setup) we get
\begin{align}\label{20210913_17:12}
F(z) & = \frac{1}{K(\xi_0, \xi_0)} \big( F(\xi_0)\, K(\xi_0, z) +  F(-\xi_0) \, K(-\xi_0, z)\big).
\end{align}
From \eqref{20210913_17:12} and \eqref{20210913_13:57} note that 
\begin{equation}\label{20210913_17:19}
zF(z) = \frac{A(z) B(\xi_0)}{\pi K(\xi_0, \xi_0)} \big( F(-\xi_0) -F(\xi_0)\big) + \frac{\xi_0 \, F(\xi_0)}{K(\xi_0, \xi_0)} K(\xi_0, z) + \frac{(-\xi_0) \, F(-\xi_0)}{K(\xi_0, \xi_0)} K(-\xi_0, z).
\end{equation}
One then sees that $zF \in \mc{H}(E)$ if and only if $F(\xi_0) = F(-\xi_0)$ (recall that $B(\xi_0) \neq 0$ since $E$ has no real zeros). Hence ${\bf A}_{E} = \xi_0$ and the extremizers have the proposed form.

\smallskip

We now consider the case $A \in \mc{H}(E)$. By \cite[Theorem 29]{Branges} a function $G \in \mc{H}(E)$ is  orthogonal to $\mc{X}(E)$ if and only if it is of the form $G(z) = uA(z) + vB(z)$ for constants $u,v \in \C$. If there exists such a function $G$ with $v \neq 0$, we find that $B \in  \mc{H}(E)$ and hence $E \in \mc{H}(E)$, contradicting \eqref{20210913_14:31}. Hence $\mc{X}(E)^{\perp} = {\rm span}\{A\}$. Recalling the discussion in \S \ref{DeBranges_Intro}, that $\Gamma_A \cup \{A\}$ is an orthogonal basis of $\mc{H}(E)$, if ${\bf 0} \neq F \in \mc{X}(E)$ we have
\begin{align}\label{20210913_16:40}
\|F\|_{\mc{H}(E)}^2  = \sum_{A(\xi) = 0}  \frac{\big| F(\xi)\big|^2 }{K(\xi, \xi)} \leq \frac{1}{\xi_0^2} \sum_{A(\xi) = 0}  \frac{|\xi|^2\,\big|F(\xi)\big|^2 }{K(\xi, \xi)} \leq \frac{1}{\xi_0^2}\, \|z F\|_{\mc{H}(E)}^2. 
\end{align}
This shows that ${\bf A}_{E} \geq \xi_0$. In order to have equality in \eqref{20210913_16:40} one must have $F(\xi) = 0$ for each $\xi \in \R$ such that $A(\xi) = 0$ and $\xi \neq \pm \xi_0$, and $zF \perp A$. By the interpolation conditions, $F$ has a representation as in \eqref{20210913_17:12}. By \eqref{20210913_17:19} one sees that $zF \perp A$ if and only if $F(\xi_0) = F(-\xi_0)$. This leads us to the same extremizers as before and to the conclusion that ${\bf A}_{E} = \xi_0$.
\end{proof}

\subsection{Proof of Theorem \ref{Thm_Db_9}} The solution of our original extremal problem \eqref{20210907_13:49} is then a corollary of Theorem \ref{DB_Thm}, applied to the Hermite-Biehler function $E = E_{G^{\sharp}, \pi\Delta}$ given by \eqref{20210909_19:00}. One has 
$${\bf A}_{G, \pi\Delta} = \xi_0,$$
where $\xi_0$ is the smallest positive real zero of the even function $A(x) := {\rm Re} \,(E_{G^{\sharp}, \pi\Delta}(x))$. Inequality \eqref{20210914_13:55} plainly leads us to Theorem \ref{Thm_Db_9}.

 \section{Appendix: Sharp embeddings} 
\subsection{Sharp constants} In this appendix, we are interested in determining the values of the sharp constants
\begin{equation*}
{\bf C}^-_{G, \pi\Delta} := \inf_{\substack{F \in \mc{H}_{\pi\Delta} \\ F\neq {\bf 0}}}  \frac{\|F\|_{L^2(\R, W_G)}}{\|F\|_{L^2(\R)}} \ \ \ {\rm and} \ \ \ {\bf C}^+_{G, \pi\Delta} := \sup_{\substack{F \in \mc{H}_{\pi\Delta} \\ F\neq {\bf 0}}}  \frac{\|F\|_{L^2(\R, W_G)}}{\|F\|_{L^2(\R)}},
\end{equation*}
associated to inequality \eqref{20210505_11:03}, and investigating the extremizing functions; see \cite[Appendix B]{CCCM} for a related problem involving Hilbert spaces of entire functions associated to the pair correlation of zeros of the Riemann zeta-function. Our study considers the cases of Theorems \ref{Thm3_PW} -- \ref{Thm_odd_orthogonal}, that is, when $G \in \{{\rm U, O}\}$ and $\Delta >0$ and when $G \in \{{\rm Sp, SO}({\rm even}), {\rm SO}(\rm{odd})\}$ and $0 < \Delta \leq 2$. The methods below could be implemented to treat the latter cases in the regime $\Delta >2$, at a higher computational cost. When $G = {\rm U}$ we have $\mathcal{H}_{{\rm U}, \pi \Delta}$ being the Paley-Wiener space $\mathcal{H}_{\pi \Delta}$ and there is nothing to do. We have then 8 sharp constants to determine, and these are described in the results below.

\begin{theorem}[Sharp constants:~orthogonal symmetry] \label{Thm12_O} For any $\Delta >0$ we have
\begin{equation*}
{\bf C}^-_{{\rm O}, \pi\Delta} = 1 \ \ \ {\rm and} \ \ \ {\bf C}^+_{{\rm O}, \pi\Delta} = \sqrt{1 + \tfrac{\Delta}{2}}\,.
\end{equation*}
\end{theorem}

\begin{theorem}[Sharp constants:~even orthogonal symmetry] \label{Thm13_SOeven}\hfill

\smallskip

\noindent {\rm (i)} For $0 < \Delta \leq 1$ we have
\begin{equation*}
{\bf C}^-_{{\rm SO(even)}, \pi\Delta} = 1 \ \ \ {\rm and} \ \ \ {\bf C}^+_{{\rm SO(even)}, \pi\Delta} = \sqrt{1 + \tfrac{\Delta}{2}}\,.
\end{equation*}

\noindent {\rm (ii)} For $1 < \Delta \leq 2$, let $\eta^+$ and $\eta^-$ be the largest and the smallest real solutions of 
\begin{align*}
\left(\tfrac{1}{2}+\tfrac{2-\Delta}{4\eta}\right)\cos\left(\tfrac{\Delta-1}{2\eta}\right)+\sin\left(\tfrac{\Delta-1}{2\eta}\right) = 1.
\end{align*}
Then 
\begin{equation*}
{\bf C}^-_{{\rm SO(even)}, \pi\Delta} = \sqrt{1 + \eta^-} \ \ \ {\rm and} \ \ \ {\bf C}^+_{{\rm SO(even)}, \pi\Delta} = \sqrt{1 + \eta^+}\,.
\end{equation*}
\end{theorem}

\begin{theorem}[Sharp constants:~symplectic symmetry] \label{Thm14_Sp}\hfill

\smallskip

\noindent {\rm (i)} For $0 < \Delta \leq 1$ we have
\begin{equation*}
{\bf C}^-_{{\rm Sp}, \pi\Delta} = \sqrt{1 - \tfrac{\Delta}{2}} \ \ \ {\rm and} \ \ \ {\bf C}^+_{{\rm Sp}, \pi\Delta} = 1.
\end{equation*}

\noindent {\rm (ii)} For $1 < \Delta \leq 2$, let $\eta^+$ and $\eta^-$ be the largest and the smallest real solutions of 
\begin{align*}
\left(\tfrac{1}{2}+\tfrac{2-\Delta}{4\eta}\right)\cos\left(\tfrac{\Delta-1}{2\eta}\right)+\sin\left(\tfrac{\Delta-1}{2\eta}\right) = 1.
\end{align*}
Then 
\begin{equation*}
{\bf C}^-_{{\rm Sp}, \pi\Delta} = \sqrt{1 - \eta^+} \ \ \ {\rm and} \ \ \ {\bf C}^+_{{\rm Sp}, \pi\Delta} = \sqrt{1 - \eta^-}\,.
\end{equation*}
\end{theorem}

\begin{theorem}[Sharp constants:~odd orthogonal symmetry] \label{Thm15_SOodd}\hfill

\smallskip

\noindent {\rm (i)} For $0 < \Delta \leq 1$ we have
\begin{equation*}
{\bf C}^-_{{\rm SO(odd)}, \pi\Delta} = 1 \ \ \ {\rm and} \ \ \ {\bf C}^+_{{\rm SO(odd)}, \pi\Delta} = \sqrt{1 + \tfrac{\Delta}{2}}\,.
\end{equation*}

\noindent {\rm (ii)} For $1 < \Delta \leq 2$, let $\eta^+$ and $\eta^-$ be the largest and the smallest real solutions of 
\begin{align*}
\left(\tfrac{3}{2}-\tfrac{2-\Delta}{4\eta}\right)\cos\left(\tfrac{\Delta-1}{2\eta}\right)-\sin\left(\tfrac{\Delta-1}{2\eta}\right) = 1.
\end{align*}
Then 
\begin{equation*}
{\bf C}^-_{{\rm SO(odd)}, \pi\Delta} = \sqrt{1 + \eta^-} \ \ \ {\rm and} \ \ \ {\bf C}^+_{{\rm SO(odd)}, \pi\Delta} = \sqrt{1 + \eta^+}\,.
\end{equation*}
\end{theorem}

\subsection{Interpolation formulas and the proofs of Theorem \ref{Thm12_O} and parts (i) of Theorems \ref{Thm13_SOeven} -- \ref{Thm15_SOodd}}  If  $F\in \mathcal{H}_{\pi\Delta}$, from the Paley-Wiener theorem and Plancherel's identity, it is a well-known fact that the norm $\|F\|_{L^2(\R)}$ can be inferred from $(1/\Delta)$-equally spaced samples as 
\begin{equation}\label{20210517_11:25}
\|F\|_{L^2(\R)}^2 = \frac{1}{\Delta}\sum_{k \in \Z} \big|F\big(\tfrac{k}{\Delta}\big)\big|^2.
\end{equation}
Moreover, the function $F$ can also be fully recovered from such samples, a classical result known as the Shannon-Whittaker interpolation formula,
 \begin{equation}\label{20210517_11:31}
F(z)= \sum_{k \in \Z} F\big(\tfrac{k}{\Delta}\big)  \dfrac{\sin\big( \pi \Delta\big(z - \tfrac{k}{\Delta}\big)\big) }{\pi \Delta\big(z- \tfrac{k}{\Delta}\big)}.
\end{equation}

When $F\in \mathcal{H}_{\pi\Delta}$ we have $\|F\|^2_{L^2(\R)} \leq \|F\|_{L^2(\R, W_{\rm O})}^2 = \|F\|^2_{L^2(\R)} + |F(0)|^2/2$, which plainly implies that ${\bf C}^-_{{\rm O}, \pi\Delta} = 1$ and that $F$ is an extremizer if and only if $F(0)=0$. On the other hand, from \eqref{20210517_11:25} we have
\begin{equation}\label{20210517_12:04}
\|F\|_{L^2(\R, W_{\rm O})}^2 = \|F\|^2_{L^2(\R)} + \frac{|F(0)|^2}{2} \leq \left( 1 + \tfrac{\Delta}{2}\right) \|F\|^2_{L^2(\R)}\,,
\end{equation}
with equality if and only if $|F(0)|^2 = \Delta \|F\|^2_{L^2(\R)}$ and $F\big(\tfrac{k}{\Delta}\big) = 0$ for $k \in \Z \setminus \{0\}$. 
This shows that ${\bf C}^+_{{\rm O}, \pi\Delta} = \sqrt{1 + \tfrac{\Delta}{2}}$, with the only extremizers, from \eqref{20210517_11:31}, being given by $F(z) = c\,\frac{\sin \pi \Delta z}{\pi \Delta z}$ with $c \in \C\setminus \{0\}$. This completes the proof of Theorem \ref{Thm12_O}.

\smallskip

Now let $0 < \Delta \leq 1$. For any $F \in \mathcal{H}_{\pi\Delta}$, letting $H(z):= F(z)\overline{F(\overline{z})}$, we have that $H$ has exponential type at most $2\pi \Delta$ and belongs to $L^1(\R)$. Using \eqref{20210506_11:40_1} and \eqref{20210506_11:42}, when $G \in \{\rm{SO(even), SO(odd)}\}$, we have
\begin{align}\label{20210517_11:59}
\|F\|_{L^2(\R, W_G)}^2 = \int_{\R} \widehat{H}(y) \,\widehat{W_{G}}(y)\,\dy = \int_{\R} \widehat{H}(y) \,\widehat{W_{{\rm O}}}(y)\,\dy  = \|F\|_{L^2(\R, W_{\rm O})}^2. 
\end{align}
The conclusion is that part (i) of Theorems \ref{Thm13_SOeven} and \ref{Thm15_SOodd} is equivalent to the already established Theorem \ref{Thm12_O} in the range $0 < \Delta \leq 1$. When $G = {\rm Sp}$ the same reasoning as in \eqref{20210517_11:59} yields
\begin{align*}
\|F\|_{L^2(\R, W_{\rm Sp})}^2 = \int_{\R} \widehat{H}(y) \,\widehat{W_{{\rm Sp}}}(y)\,\dy =  \int_{\R} |F(x)|^2 \big(1 - \tfrac{1}{2} \boldsymbol{\delta}_0(x) \big) \,\dx = \|F\|^2_{L^2(\R)} - \frac{|F(0)|^2}{2}.
\end{align*}
As in \eqref{20210517_12:04}, this leads us to ${\bf C}^+_{{\rm Sp}, \pi\Delta} = 1$, with $F$ being an extremizer if and only if $F(0) =0$, and ${\bf C}^-_{{\rm Sp}, \pi\Delta} = \sqrt{1 - \tfrac{\Delta}{2}}$, with $F$ being an extremizer if and only if $F(z) = c\,\frac{\sin \pi \Delta z}{\pi \Delta z}$ with $c \in \C\setminus \{0\}$. This concludes the proof of part (i) of Theorem \ref{Thm14_Sp}.

\subsection{Extremal eigenvalues} \label{Extr_eigen} For part (ii) of Theorems \ref{Thm13_SOeven} -- \ref{Thm15_SOodd} we take a slightly different path, bringing in elements from the theory of compact and self-adjoint operators as already done in \S \ref{SubS_Interlude}. Throughout this subsection, assume that $G \in \{{\rm Sp, SO}({\rm even}), {\rm SO}(\rm{odd})\}$ and $\Delta >1$. Let us write $W_G$ in \eqref{densities} as 
$$W_G(x) = 1 + \Phi_G(x).$$
Hence, if $F \in \mathcal{H}_{\pi \Delta}$, we have
$$\|F\|_{L^2(\R, W_G)}^2 = \|F\|_{L^2(\R)}^2 + \int_{\R} |F(x)|^2\, \Phi_G(x)\,\dx.$$
From now on let $I := \big[\!-\!\frac{\Delta}{2}, \frac{\Delta}{2}\big]$. By Plancherel's theorem, note that 
\begin{align*}
\int_{\R} |F(x)|^2\, \Phi_G(x)\,\dx = \int_{\R} \big(\widehat{F} * \widehat{\Phi}_G\big)(y) \,\overline{\widehat{F}(y)}\,\dy = \langle \,T_G (\widehat{F}) \,,\, \widehat{F}\,\rangle_{L^2(I)},
\end{align*}
where $T_G: L^2(I) \to L^2(I)$ is the operator defined by 
\begin{equation}\label{20210519_10:53}
(T_G\,u )(y) = \int_I  \widehat{\Phi}_G(y-s)\,u(s)\,\ds \ \ \ {\rm for} \ \ y \in I.
\end{equation}
Note that $T_G$ is a self-adjoint operator, i.e. $\langle T_G u , v \rangle_{L^2(I)} = \langle u , T_G v \rangle_{L^2(I)}$ for any $u, v \in L^2(I)$. Also, since the kernel $(y,s) \mapsto \widehat{\Phi}_G(y-s)$ belongs to $L^2(I \times I)$, $T_G$ is a Hilbert-Schmidt operator and hence compact. One should always keep in mind that, from the Paley-Wiener theorem, the map $F \to \widehat{F}$ is an isometry between $\mathcal{H}_{\pi \Delta}$ and $L^2(I)$. Defining 
\begin{equation}\label{20210519_10:41}
{\bf L}^-_{G, \pi\Delta} := \inf_{\substack{F \in \mc{H}_{\pi\Delta} \\ F\neq {\bf 0}}}  \frac{\int_{\R} |F(x)|^2\, \Phi_G(x)\,\dx}{\|F\|_{L^2(\R)}^2}  = \inf_{\substack{\widehat{F} \in L^2(I) \\ \widehat{F}\neq {\bf 0}}} \frac{\langle \,T_G (\widehat{F}) \,,\, \widehat{F}\,\rangle_{L^2(I)}}{\|\widehat{F}\|_{L^2(I)}^2} 
\end{equation}
and
\begin{equation}\label{20210519_10:42}
{\bf L}^+_{G, \pi\Delta} := \sup_{\substack{F \in \mc{H}_{\pi\Delta} \\ F\neq {\bf 0}}}  \frac{\int_{\R} |F(x)|^2\, \Phi_G(x)\,\dx}{\|F\|_{L^2(\R)}^2} = \sup_{\substack{\widehat{F} \in L^2(I) \\ \widehat{F}\neq {\bf 0}}} \frac{\langle \,T_G (\widehat{F}) \,,\, \widehat{F}\,\rangle_{L^2(I)}}{\|\widehat{F}\|_{L^2(I)}^2} \,,
\end{equation}
it is clear that 
$${\bf C}^-_{G, \pi\Delta} = \sqrt{1 +{\bf L}^-_{G, \pi\Delta}} \ \ \ {\rm and} \ \ \ {\bf C}^+_{G, \pi\Delta} = \sqrt{1 + {\bf L}^+_{G, \pi\Delta}}\ .$$

\smallskip

Assume for a moment that we have established the claim that 
\begin{equation}\label{20210518_09:24}
{\bf L}^{-}_{G, \pi\Delta} <0 \ \ \ {\rm and} \ \ \ {\bf L}^{+}_{G, \pi\Delta} >0.
\end{equation}
From the classical theory of compact and self-adjoint operators, e.g.  \cite[Theorem 6.8 and Proposition 6.9]{Brezis}, the infimum and supremum in \eqref{20210519_10:41} and \eqref{20210519_10:42} are attained by eigenfunctions $\widehat{F} \in L^2(I)$ of $T_G$.

\subsection{Non-trivial signs} We now verify the claim \eqref{20210518_09:24} for $G \in \{{\rm Sp, SO}({\rm even}),$ $ {\rm SO}(\rm{odd})\}$ and $\Delta >1$. Let $F(z) = \tfrac{\sqrt{\Delta}\sin \pi \Delta z}{\pi \Delta z} \in \mathcal{H}_{\pi \Delta}$ and note that $\|F\|_{L^2(\R)} = 1$. By Plancherel's theorem, for any $t \in \R$,  
\begin{align}\label{20210517_16:14}
\begin{split}
\int_{\R} |F(x-t)|^2\, \left( \tfrac{\sin 2 \pi x}{2 \pi x}\right)\dx & = \frac{1}{2} \int_{\R} \max\big\{1 - \tfrac{|y|}{\Delta}, 0\big\} \,{\bf 1}_{[-1,1]}(y)\,e^{-2 \pi i y t} \, \dy \\
& = \frac{(\Delta -1)\pi t \sin(2\pi t) + \sin^2(\pi t)}{2\Delta (\pi t)^2}.
\end{split}
\end{align}
For $t =0$, the right-hand side of \eqref{20210517_16:14} yields $1 - \tfrac{1}{2\Delta}$ and hence
$$- {\bf L}^{-}_{{\rm Sp}, \pi\Delta} = {\bf L}^{+}_{{\rm SO(even)}, \pi\Delta} \geq 1 - \tfrac{1}{2\Delta} >0.$$
On the other hand, taking $t = k + \tfrac{3}{4}$ with $k \geq 0$ a large integer, we see that the numerator of \eqref{20210517_16:14} eventually becomes strictly negative. This is enough to conclude that 
$$- {\bf L}^{+}_{{\rm Sp}, \pi\Delta} = {\bf L}^{-}_{{\rm SO(even)}, \pi\Delta} < 0.$$

Similarly, as in \eqref{20210517_16:14}, we have
\begin{equation}\label{20210517_16:38}
\int_{\R} |F(x-t)|^2\, \left( -\tfrac{\sin 2 \pi x}{2 \pi x} + \boldsymbol{\delta}_0(x)\right)\dx=  \frac{-(\Delta -1)\pi t \sin(2\pi t) - \sin^2(\pi t) + 2\sin^2(\pi \Delta t)}{2\Delta (\pi t)^2}.
\end{equation}
For $t =0$, the right-hand side of \eqref{20210517_16:38} becomes $\Delta - 1 + \tfrac{1}{2\Delta}$ and hence 
$${\bf L}^{+}_{{\rm SO(odd)}, \pi\Delta} \geq \Delta - 1 + \tfrac{1}{2\Delta} >0.$$
On the other hand, taking $t = k + \tfrac{1}{4}$ with $k \geq 0$ a large integer, we see that the numerator of \eqref{20210517_16:38} eventually becomes strictly negative. Hence
$${\bf L}^{-}_{{\rm SO(odd)}, \pi\Delta} < 0.$$

\subsection{Proofs of parts (ii) of Theorems \ref{Thm13_SOeven} and \ref{Thm14_Sp}: finding $ {\bf L}^{\pm}_{{\rm SO(even)}, \pi\Delta} = - {\bf L}^{\mp}_{{\rm Sp}, \pi\Delta}$} \label{Sub5.5} Assume now that $G = {\rm SO(even)}$ and $1 < \Delta \leq 2$. From the discussion in \S \ref{Extr_eigen}, we must find the extremal eigenvalues of the compact operator $T_G$ defined in \eqref{20210519_10:53}. Letting $u = \widehat{F}$, we must solve the functional equation (recall that $I := \big[\!-\!\frac{\Delta}{2}, \frac{\Delta}{2}\big]$)
 \begin{equation}\label{20210518_11:22}
 \dfrac{1}{2}\int_{y-1}^{y+1}u(s)\,\ds = \eta \, u(y)  
 \end{equation}
in $L^2(I)$, with $\eta \neq 0$. This is a challenge similar in spirit to what we have faced in \S \ref{SubS_Interlude}, \S \ref{Subs_Delta=2}, and \S \ref{subsec_Sp_1-2}. As before, from \eqref{20210518_11:22} we may assume without loss of generality that $u$ is absolutely continuous on $I$, and that \eqref{20210518_11:22} holds pointwise everywhere on $I$. By the fundamental theorem of calculus, we then get
\begin{align} \label{20210518_13:19}
\begin{split}
\eta\, u'(y) + \tfrac12 u(y-1) &= 0 \ \ \ {\rm for}  \ \ 1 - \tfrac{\Delta}{2} < y <  \tfrac{\Delta}{2};\\
\eta\, u'(y)  &= 0 \ \ \ {\rm for}  \ \  \tfrac{\Delta}{2}-1 < y <  1-\tfrac{\Delta}{2};  \\
\eta\, u'(y) - \tfrac12 u(y+1) &= 0 \ \ \ {\rm for}  \ \ -\tfrac{\Delta}{2} < y <  \tfrac{\Delta}{2}-1,
\end{split}
\end{align}
where the second equation above can be disregarded if $\Delta =2$. Manipulating these equations, as in \S \ref{Subs_Delta=2}, we are led to the general solution
\begin{align}\label{20210518_13:20}
u(y) = \left\{
\begin{array}{lll}
A\,e^{iy/2\eta}+B\,e^{-iy/2\eta} & {\rm for}  &1 - \tfrac{\Delta}{2} \leq y \leq  \tfrac{\Delta}{2};\\
D &  {\rm for} &  \tfrac{\Delta}{2}-1 \leq y \leq  1-\tfrac{\Delta}{2};\\
A_1\,e^{iy/2\eta}+B_1\,e^{-iy/2\eta} &  {\rm for} &-\tfrac{\Delta}{2} \leq y \leq  \tfrac{\Delta}{2}-1,
\end{array}
\right.
\end{align}
where $A,B,A_1, B_1, D \in \C$. In the philosophy of \S \ref{Subs_Delta=2}, these were the {\it descending steps}. We now proceed to our {\it ascending steps} to figure out the constants. Plugging \eqref{20210518_13:20} into the first equation of \eqref{20210518_13:19} we get 
\begin{equation}\label{20210519_15:39}
A_1 = - i\,e^{i/2\eta} A  \ \ \ {\rm and} \ \  \ B_1 =\,i\,e^{-i/2\eta} B.
\end{equation} 
At this point we have to find the constants $A,B,D$ (not all zero) and the unknown extremal eigenvalue $\eta \neq 0$ in such a way that our function $u$ is continuous on $I$ and verifies \eqref{20210518_11:22} at all points $y \in I$. One can show that a necessary and sufficient condition is given by the following system of equations:
\begin{align}\label{20210518_13:40}
\begin{cases}
a_1 A  + \overline{a_1} \,B  + \left( \tfrac{2 - \Delta}{2}\right) D= 0 \,;\\
b_1 A  + \overline{b_1} \,B + \left( \tfrac{2 - \Delta}{2}\right) D= 0\,;\\
\tau A + \overline{\tau} B - D= 0,
\end{cases}
\end{align}
with $a_1 := \eta\big( -i \,e^{\Delta i / 4 \eta} -  e^{\Delta i / 4 \eta} + i \,e^{(2-\Delta) i / 4 \eta}\big)$, $b_1 := \eta \big( -e^{\Delta i / 4 \eta} + e^{(2-\Delta) i / 4 \eta} + i \,e^{(2-\Delta) i / 4 \eta}\big)$, and $\tau := - i\, e^{\Delta i / 4 \eta}$. The first two equations in \eqref{20210518_13:40} come from the evaluation of \eqref{20210518_11:22} at the points $y = \tfrac{\Delta}{2}$ and $y =- \tfrac{\Delta}{2}$, while the third one (which is not necessary if $\Delta =2$) comes from the continuity of $u$ at the point $y = \tfrac{\Delta}{2} -1$. In \eqref{20210518_13:40}, multiplying the third equation by $ \left( \tfrac{2 - \Delta}{2}\right)$ and adding it up to the first two equations yields
\begin{align}\label{20210518_13:46}
\begin{cases}
a A + \overline{a} \,B = 0 \,;\\
b A + \overline{b} \,B = 0\,,
\end{cases}
\end{align}
with $a := a_1 + \left( \tfrac{2 - \Delta}{2}\right)\tau$ and $b := b_1 + \left( \tfrac{2 - \Delta}{2}\right)\tau$.

\smallskip

If $a\overline{b} - \overline{a}b \neq 0$, from \eqref{20210518_13:46} we would get $A= B = 0$, which would ultimately imply that our solution $u = {\bf 0}$, a contradiction. Hence we must have $a\overline{b} - \overline{a}b = 0$. This condition is equivalent to
\begin{align}\label{20210518_16:09}
\left(\tfrac{1}{2}+\tfrac{2-\Delta}{4\eta}\right)\cos\left(\tfrac{\Delta-1}{2\eta}\right)+\sin\left(\tfrac{\Delta-1}{2\eta}\right) = 1.
\end{align}
Hence our desired values of $\eta$ are the largest solution of \eqref{20210518_16:09} (which, in particular, verifies $0 < \eta <1$), and the smallest solution of \eqref{20210518_16:09} (which, in particular, verifies $-1 < \eta <0$).

\smallskip

A few words on the extremizers. If $(a,b) \neq (0,0)$, from \eqref{20210518_13:46} we get $B$ in terms of $A$, and from the third equation in \eqref{20210518_13:40} we get $D$ in terms of $A$. This determines $u$ uniquely (modulo multiplication by a complex constant $A$), which is the same as saying that the associated eigenspace has dimension $1$. This is what generally happens, with only one exception that we now describe. In order to have $a=b=0$, one must have $\eta = (\Delta-2)/2 < 0$ and $\sin\frac{(\Delta - 1)}{2\eta} = \sin\frac{(\Delta - 1)}{(\Delta-2)} = 1$, i.e. $\Delta = (1 - \pi + 4k\pi)/(1 - \frac{\pi}{2} + 2k\pi)$ for $k \in \N$. If $k=1$, i.e. when $\Delta = (1 + 3\pi) / (1 + \frac{3\pi}{2})$, one indeed has $\eta = (\Delta-2)/2$ being the smallest solution of \eqref{20210518_16:09}, and this is the one exception where the extremal eigenspace has dimension $2$.  If $\Delta = (1 - \pi + 4k\pi)/(1 - \frac{\pi}{2} + 2k\pi)$ for $k \geq 2$, then $\eta = (\Delta-2)/2$ is still a solution of \eqref{20210518_16:09}, but not the smallest one since $\eta^*  = (1- \Delta)/(3\pi)$ is a strictly smaller one.

\subsection{Proof of part (ii) of Theorem \ref{Thm15_SOodd}: finding $ {\bf L}^{\pm}_{{\rm SO(odd)}, \pi\Delta}$} The idea here is the same as in the previous subsection. Let $G = {\rm SO(odd)}$ and $1 < \Delta \leq 2$. We must solve the functional equation
 \begin{equation}\label{20210519_15:25}
 -\dfrac{1}{2}\int_{y-1}^{y+1}u(s)\,\ds +  \int_{I}u(s)\,\ds= \eta \, u(y)  
 \end{equation}
in $L^2(I)$, with $\eta \neq 0$. We may assume that $u$ is absolutely continuous on $I$, and that \eqref{20210519_15:25} holds pointwise everywhere on $I$. Proceeding as in \eqref{20210518_13:19} - \eqref{20210519_15:39} we arrive at
\begin{align*}
u(y) = \left\{
\begin{array}{lll}
A\,e^{iy/2\eta}+B\,e^{-iy/2\eta} & {\rm for}  &1 - \tfrac{\Delta}{2} \leq y \leq  \tfrac{\Delta}{2};\\
D &  {\rm for} &  \tfrac{\Delta}{2}-1 \leq y \leq  1-\tfrac{\Delta}{2};\\
 i\,e^{i/2\eta}\,A\,e^{iy/2\eta}- i\,e^{-i/2\eta}\,B\,e^{-iy/2\eta} &  {\rm for} &-\tfrac{\Delta}{2} \leq y \leq  \tfrac{\Delta}{2}-1,
\end{array}
\right.
\end{align*}
with constants $A,B,D$ (not all zero) and the unknown extremal eigenvalue $\eta \neq 0$ to be found. We are then led to the following system of equations (which is necessary and sufficient):
\begin{align}\label{20210519_15:47}
\begin{cases}
a_1 A  + \overline{a_1} \,B  + \left( \tfrac{2 - \Delta}{2}\right) D= 0 \,;\\
b_1 A  + \overline{b_1} \,B + \left( \tfrac{2 - \Delta}{2}\right) D= 0\,;\\
\tau A + \overline{\tau} B - D= 0,
\end{cases}
\end{align}
with $a_1 := \eta\big( -i \,e^{\Delta i / 4 \eta} + e^{\Delta i / 4 \eta} + i \,e^{(2-\Delta) i / 4 \eta} - 2e^{(2-\Delta) i / 4 \eta}\big)$, $b_1 := \eta\big( -2i \,e^{\Delta i / 4 \eta} + e^{\Delta i / 4 \eta} + i \,e^{(2-\Delta) i / 4 \eta} - e^{(2-\Delta) i / 4 \eta}\big)$, and $\tau :=  i\, e^{\Delta i / 4 \eta}$. The first two equations in \eqref{20210519_15:47} come from the evaluation of \eqref{20210519_15:25} at the points $y = \tfrac{\Delta}{2}$ and $y =- \tfrac{\Delta}{2}$, while the third one (which is not necessary if $\Delta =2$) comes from the continuity of $u$ at the point $y = \tfrac{\Delta}{2} -1$. In \eqref{20210519_15:47}, multiplying the third equation by $ \left( \tfrac{2 - \Delta}{2}\right)$ and adding it up to the first two equations yields
\begin{align*}
\begin{cases}
a A + \overline{a} \,B = 0 \,;\\
b A + \overline{b} \,B = 0\,,
\end{cases}
\end{align*}
with $a := a_1 + \left( \tfrac{2 - \Delta}{2}\right)\tau$ and $b := b_1 + \left( \tfrac{2 - \Delta}{2}\right)\tau$. We must then have $a\overline{b} - \overline{a}b = 0$, which is equivalent to
\begin{align}\label{20210519_16:44}
\left(\tfrac{3}{2}-\tfrac{2-\Delta}{4\eta}\right)\cos\left(\tfrac{\Delta-1}{2\eta}\right)-\sin\left(\tfrac{\Delta-1}{2\eta}\right) = 1.
\end{align}
Hence our desired values of $\eta$ are the largest solution of \eqref{20210519_16:44} and the smallest solution of \eqref{20210519_16:44} (which, in particular, verifies $-1 < \eta <0$). The corresponding extremal function is unique (modulo multiplication by a complex constant) in all cases. Indeed, one can check that the only possibility to have $a=b=0$ is to have $\eta = (2-\Delta)/6 > 0$ and $\sin\frac{(\Delta - 1)}{2\eta} = \sin\frac{3(\Delta - 1)}{(2- \Delta)}= -1$, i.e. $\Delta = (3 -\pi  + 4 k \pi )/(3 -\frac{\pi}{2} + 2k \pi)$ for $k \in \N$. When $k=1$, i.e. when $\Delta = (3 + 3\pi)/(3 + \frac{3\pi}{2})$, we have that $\eta =  (2-\Delta)/6 = 1/(6 + 3\pi)$ is a solution of \eqref{20210519_16:44} but one can verify numerically that there is a larger one around $0.93303\ldots$. If $\Delta = (3 -\pi  + 4 k \pi )/(3 -\frac{\pi}{2} + 2k \pi)$ for $k \geq 2$, then $\eta =  (2-\Delta)/6>0$ is still a solution of \eqref{20210519_16:44} but it is not the largest one since $\eta^* = (\Delta - 1)/(3\pi)$ is a strictly larger one.

\section*{Acknowledgments}
E.C. acknowledges support from FAPERJ - Brazil.  A.C. was supported by Grant $275113$ of the Research Council of Norway. M.B.M. was supported in part by the Simons Foundation (award 712898) and the National Science Foundation (DMS-2101912). We are thankful to the referee and Emily Quesada-Herrera for helpful comments and discussions, and to Guillaume Ricotta and Emmanuel Royer for bringing reference \cite{Bernard} to our attention.

\end{document}